\documentclass[11pt]{article}
    \usepackage{url}
    \usepackage{verbatim}
    \usepackage[titletoc]{appendix}
    \usepackage{graphicx}
	\usepackage{amsmath}
	\usepackage{amsthm}
	\usepackage{amsfonts}
	\usepackage{graphicx}
    \usepackage{nicefrac}
    \usepackage{longtable}
    \usepackage{color}
    \usepackage{graphicx, amssymb, subcaption,graphics}

    \newcommand{\be}{\begin{equation}}
    \newcommand{\ee}{\end{equation}}

    \newcommand{\nrm}[1]{\left\| #1 \right\|}

    \newcommand\dt {{\Delta t}}

    \def\x{\mbox{\boldmath $x$}}
        
	\newtheorem{thm}{Theorem}[section]
	\newtheorem{prop}[thm]{Proposition}
	\newtheorem{cor}[thm]{Corollary}
	
	\newtheorem{lem}[thm]{Lemma}
	\newtheorem{rem}[thm]{Remark}
	\newtheorem{defi}[thm]{Definition}
	
\begin{document}
	\title{An Energy Stable BDF2 Fourier Pseudo-Spectral Numerical Scheme for the Square Phase Field Crystal Equation}
	
	\author{
Kelong Cheng \thanks{School of Science, Southwest University of Science and Technology,  Mianyang, Sichuan 621010, P. R. China (zhengkelong@swust.edu.cn)}
	\and		
Cheng Wang\thanks{Department of Mathematics, The University of Massachusetts, North Dartmouth, MA  02747}	
	\and
Steven M. Wise\thanks{Department of Mathematics, The University of Tennessee, Knoxville, TN 37996 (Corresponding author: swise1@utk.edu)} 
	}

	\maketitle
	\numberwithin{equation}{section}

	\begin{abstract}
In this paper we propose and analyze an energy stable numerical scheme for the square phase field crystal (SPFC) equation, a gradient flow modeling crystal dynamics at the atomic scale in space but on diffusive scales in time. In particular, a modification of the free energy potential to the standard phase field crystal model leads to a composition of the 4-Laplacian and the regular Laplacian operators. To overcome the difficulties associated with this highly nonlinear operator, we design  numerical algorithms based on the structures of the individual energy terms. A Fourier pseudo-spectral approximation is taken in space, in such a way that the energy structure is respected, and summation-by-parts formulae enable us to study the discrete energy stability for such a high-order spatial discretization.  In the temporal approximation, a second order BDF stencil is applied, combined with an appropriate extrapolation for the concave diffusion term(s). A second order artificial Douglas-Dupont-type regularization term is added to ensure energy stability, and a careful analysis leads to the artificial linear diffusion coming at an order lower that that of surface diffusion term. Such a choice leads to reduced numerical dissipation. At a theoretical level, the unique solvability, energy stability are established, and an optimal rate convergence analysis is derived in the $\ell^\infty (0,T; \ell^2) \cap \ell^2 (0,T; H_N^3)$ norm. In the numerical implementation, the preconditioned steepest descent (PSD) iteration is applied to solve for the composition of the highly nonlinear 4-Laplacian term and the standard Laplacian term, and a geometric convergence is assured for such an iteration. Finally, a few numerical experiments are presented, which confirm the robustness and accuracy of the proposed scheme. 
	\end{abstract}
	
	
\medskip
	

	\section{Introduction}
	
The phase field crystal (PFC) model was proposed in \cite{elder02} as a new approach to simulating crystal dynamics at the atomic scale in space but on diffusive scales in time.  The model naturally incorporates elastic and plastic deformations, multiple crystal orientations and defects and has already been used to simulate a wide variety of microstructures, such as epitaxial thin film growth~\cite{elder04}, grain growth~\cite{stefanovic06}, eutectic solidification~\cite{elder07}, and dislocation formation and motion~\cite{stefanovic06}, to name a few.  See the related review~\cite{provatas07}. The key idea is that the phase variable describes a coarse-grained temporal average of the number density of atoms and the approach can be related to dynamic density functional theory \cite{backofen07, marconi99}.  This method represents a significant advantage over other atomistic methods, such as molecular dynamics methods where the time steps are constrained by atomic-vibration time scales.	 In more detail, the dimensionless energy is given by the following form~\cite{elder02, elder04, swift77}
	\begin{equation}
E_{\rm pfc}(\phi) = \int_\Omega \left\{ \frac{1}{4}\phi^4 +\frac{1-\varepsilon}{2}\phi^2-\left|\nabla\phi\right|^2 +\frac{1}{2} (\Delta\phi)^2 \right\} d{\bf x}  , 
  \quad 0 < \varepsilon , 
	\label{energy-PFC}
	\end{equation} 
where $\Omega\subset \mathbb{R}^D$, $D = 2$ or 3, $\phi:\Omega\rightarrow \mathbb{R}$ is the atom density field. Typically the parameter $\varepsilon$ which represents a deviation from the melting temperature, satisfies $0< \varepsilon < 1$, though it may be possible that $\varepsilon>1$. In this paper, $\phi$ is assumed to be periodic on the rectangular domain $\Omega$. Quite often, in the physics literature especially, the energy is rewritten as
	\begin{equation}
E_{\rm pfc}(\phi) = \int_\Omega \left\{ \frac{1}{4}\phi^4  - \frac{\varepsilon}{2}\phi^2 +\frac{1}{2}\phi\left(1+\Delta\right)^2 \phi \right\} d{\bf x},
	\label{energy-PFC-alt-1}
	\end{equation} 
where
	\[
\left(1+\Delta\right)^2 \phi =  \left(1+2\Delta+\Delta^2\right) \phi = \phi+2\Delta\phi +\Delta^2\phi .
	\]
The two views of the energy allow us to analyze the convexity structure from different perspectives. In \eqref{energy-PFC}, we view the quadratic term $-\left|\nabla\phi\right|^2$ as destabilizing (concave), and $\frac{1-\varepsilon}{2}\phi^2$ (and all other terms) as stabilizing (convex). This view is valid only when $a:=1-\varepsilon$ is positive, which, as we have said, may be slightly restrictive.  By contrast, in \eqref{energy-PFC-alt-1}, we view $- \frac{\varepsilon}{2}\phi^2$ as destabilizing and all other terms as stabilizing. In either case, the phase field crystal (PFC) equation is defined as
	\[
\partial_t \phi = \Delta\mu, \quad \mu := \delta_\phi E_{\rm pfc} = \phi^3 -\varepsilon\phi +(1+\Delta)^2\phi . 	
	\]
When $\varepsilon<0$, the PFC equation can have solutions that exhibit spatial oscillations, and, typically in 2D, the peaks and valleys of $\phi$ are arranged in a hexagonal pattern. Such solutions are considered to represent ``solid phase" solutions in the model. It is also possible to have ``liquid phase" solutions, which are spatially uniform and constant; and these can even be in coexistence with the solid phase solutions to describe a crystal in equilibrium with its melt. See, for example, the book by Provatas and Elder~\cite{provatas10}.

Alternate lattice structures, such as ``square" symmetry crystal lattices, are possible in 2D solutions. As mentioned in \cite{elder04}, and motivated by the work of~\cite{golovin03}, a different choice of nonlinear term in the PFC model is needed to obtain a square symmetry crystal lattice rather than the usual hexagonal structure.   Specifically, such symmetries can be obtained~\cite{golovin03} by replacing $\phi^4$ in (\ref{energy-PFC}) with $\left| \nabla \phi \right|^4$. (See also~\cite{wu10} for a related method.)  Doing so one obtains an energy of the form  
	\begin{align}
E_{\rm spfc}(\phi) & = \int_\Omega\left\{ \frac{a}{2}\phi^2 +\frac{1}{4} \left| \nabla \phi \right|^4  -\left|\nabla\phi\right|^2 + \frac12 (\Delta\phi)^2 \right\} d{\bf x} 
	\nonumber
	\\
& =  \int_\Omega\left\{ \frac{1}{4} \left| \nabla \phi \right|^4 -\frac{\varepsilon}{2}\phi^2 +  \frac12\phi \left(1+\Delta\right)^2 \phi  \right\} d{\bf x} .
	\label{energy-SPFC} 
	\end{align} 
We observe that there are essential similarities between this energy and the Aviles-Giga-type energy~\cite{aviles96}.  The square phase field crystal (SPFC) equation is given by the following dynamics
	\begin{equation}
	\partial_t \phi =  \Delta \mu \ ,   \quad  \mu := \delta_\phi E_{\rm spfc} =  - \nabla \cdot \left(  | \nabla \phi |^2  \nabla \phi \right) -\varepsilon \phi + \left(1+ \Delta\right)^2 \phi .
  	\label{equation-SPFC}
	\end{equation}
We will  assume for simplicity that $a=1-\varepsilon>0$. In this case, the energy will be bounded from below, a fact that we will prove later. For the standard PFC model and its modified version, there have been extensive numerical works~\cite{baskaran13a, baskaran13b, dong18, hu09, wang10c, wang11a, wise09, zhang13}, \emph{et cetera}, in the existing literature.  Of course, because of its generality, the new SAV approach of Shen \emph{et al.}~\cite{shen18a} could be applied to the PFC and SPFC problems.  For the SPFC equation, few if any simulation results exist, to our knowledge, though a closely related equation is solved in~\cite{golovin03}.

\textcolor{black}{Regarding scaling and parameters, we first note that $\epsilon$ -- which represents the deviation of the temperature from the freezing temperature -- in the present setting is not necessarily a small parameter. However, in the standard PFC literature, the characteristic feature size, that is the unit cell size of the pattern, is generally taken to be $O(1)$ and the domain size, $L$, is typically large, perhaps several hundred unit cell sizes wide. If one rescales the domain so that $L = O(1)$, then the feature size is a small fraction of 1, and a small parameter appears in the equation:
	\[
\partial_\tau \phi = \Delta_z\left( - \sigma^4 \nabla_z \cdot \left(  | \nabla_z \phi |^2  \nabla_z \phi \right) -\varepsilon \phi + \left(1+ \sigma^2\Delta_z\right)^2 \phi\right) , 
	\]
where $\tau$ and $z$ are rescaled time and space coordinates and $\sigma>0$ is a potentially small parameter. In particular, we see that the highest order diffusion term $\sigma^4\Delta_z^2\phi$ and the nonlinear term $- \sigma^4 \nabla_z \cdot \left(  | \nabla_z \phi |^2  \nabla_z \phi \right)$ will always be of the same magnitude, so that the stiffness in the nonlinear term could never be removed via scaling. Implicit treatment of the highest order surface diffusion as well as the nonlinear term is necessary to ensure an energy stability at a theoretical level.}

\textcolor{black}{Regarding the nonlinearity, although the only difference between the standard PFC and SPFC equations is the replacement of $\phi^4$ by $| \nabla \phi |^4$ in the free energy functional, the analysis and numerical approximation of the later are much more challenging, especially when using pseudo-spectral approximations of spatial derivatives. For example, a nonlinear multigrid numerical solver has been very successfully applied to the standard PFC model in the framework of finite differences~\cite{baskaran13a, hu09, wise09}. However, such a solver leads to a fairly poor numerical performance in the computation of the 4-Laplacian problem, due to its highly nonlinear nature. In addition, the $H^{-1}$ gradient flow pattern for the 4-Laplacian term makes the overall system even more difficult to solve. All these well-known difficulties have made a rigorous numerical analysis for the SPFC model an open problem until now.} 

In this article we consider two energy stable schemes for the SPFC equation, both with second order temporal accuracy and Fourier pseudo-spectral spatial discretization. They will differ in construction based only on whether one views $-\left|\nabla\phi\right|^2$ as destabilizing or one views $-\frac{\varepsilon}{2}\phi^2$ as destabilizing. Since the schemes are closely related, we will conduct the detailed convergence analysis only for one and just comment here and there on the differences one would encounter in the analysis of the other. 

The SPFC equation~\eqref{equation-SPFC} is structurally very different from the PFC model, with a much higher nonlinearity.  However, its energy~\eqref{energy-SPFC} can still be decomposed into purely convex and concave parts, so that a first order convex splitting scheme could be appropriately derived. On the other hand, first order temporal accuracy -- especially in the convex splitting setting which is highly dissipative -- is not satisfactory in practical computations, in particular due to the large time scale involved for the SPFC model. Therefore, a second-order-accurate in time, energy stable numerical algorithm is highly desired. Instead of the modified Crank-Nicolson approach for the gradient structure, which has been successfully applied to the Cahn-Hilliard~\cite{chen19a, cheng2018a, cheng16a, diegel17, diegel16, du91, guan17a, guan14b, guo16, han15} and epitaxial thin film equation~\cite{shen12}, we make use of a modified backward differentiation formula (BDF) approach. In more details, we apply the second order BDF concept to derive second order temporal accuracy, but modified so that the concave diffusion term is treated by an explicit extrapolation. Such an explicit treatment for the concave part of the chemical potential ensures the unique solvability of the scheme without sacrificing energy stability. An additional term $A \dt \Delta (\phi^{k+1} - \phi^k)$ is added, which represents a second order Douglas-Dupont-type regularization. Moreover, a careful analysis shows that energy stability is guaranteed for the proposed numerical schemes, provided a mild condition is enforced:
	\[
A \ge  \frac{\varepsilon^2}{16} . 
	\]
See also a related discussion in~\cite{cheng2018c} for the epitaxial thin film growth model. 

In the spatial discretization, we use Fourier pseudo-spectral approximation for its ability to capture more detailed structure with a reduced computational cost. Summation-by-parts formulae and aliasing error control techniques enable us to derive unique solvability and energy stability for the fully discrete numerical scheme. As a result of this discrete energy stability, a uniform-in-time discrete $H^2$ bound for the numerical solution becomes available. With such an $H_N^2$ bound at hand, we are able to derive a discrete $W_N^{1,6}$ bound for the numerical solution, uniform-in-time, with the help of discrete Sobolev embedding in the Fourier pseudo-spectral space. Such an embedding analysis cannot be derived from a straightforward calculation.  Instead, detailed discrete Fourier analyses, combined with certain non-trivial aliasing error estimates, are involved in the derivation. In turn, these preliminary estimates enable one to obtain an optimal rate ($O (\dt^2 + h^m)$) convergence analysis for the proposed numerical scheme, in the $\ell^\infty (0,T; \ell^2) \cap \ell^2 (0,T; H_N^3)$ norm. 

There have been related works of spectral/pseudo-spectral differentiation to various gradient flow models; see the related references~\cite{cheng16a, LiD2016b}, \emph{et cetera}. On the other hand, such a differentiation turns out to be a global operator in space, and this feature leads to great challenges in the numerical implementations, especially in the case of an implicit treatment of nonlinear terms. Moreover, due to the composition of the highly nonlinear 4-Laplacian and the regular Laplacian operators, the numerical implementation becomes even more challenging. As a result, an efficient solver for a regularized p-Laplacian equation, in the discrete $H^{-1}$ space, has to be utilized. In a recent work~\cite{feng2017preconditioned}, a preconditioned steepest descent (PSD) algorithm was proposed for such problems. At each iteration stage, only a purely linear elliptic equation needs to be solved to obtain a search direction, and the numerical efficiency for such an elliptic equation could be greatly improved with the use of FFT-based solvers. In turn, an optimization in the given search direction becomes one-dimensional, with its well-posedness assured by convexity arguments. Moreover, a geometric convergence of such an iteration could be theoretically derived, for both the $L^2$ and $H^{-1}$ gradient flow structures, so that a great improvement of the numerical efficiency is justified. For the $L^2$ gradient flow, the numerical comparison has been reported in~\cite{fengW17c} for the epitaxial thin film model, with an application of the Polak-Ribi\'ere variant of NCG (nonlinear conjugate gradient) method \cite{polak69}, reported in \cite{shen12, wang10a}. For the 4-Laplacian problem in an $H^{-1}$ gradient flow, the PSD solver in the finite difference version has been reported in~\cite{fengW17b} for the functionalized Cahn-Hilliard/Willmore model, while an efficient spectral/pseudo-spectral PSD solver for the composition of 4-Laplacian and regular Laplacian operators has not been available in the existing literature.   
 	
 The outline of the paper is given as follows. In Section~\ref{sec:numerical scheme} we present the numerical scheme. First we review the Fourier pseudo-spectral approximation in space and recall an aliasing error control technique. Then we formulate the proposed numerical scheme. Subsequently, the unique solvability and energy stability analyses are provided in Section~\ref{sec:stability}, and an optimal rate convergence analysis is established in Section~\ref{sec:convergence}. In addition, the details of the preconditioned steepest descent (PSD) solver are outlined in Section~\ref{sec:PSD}. Some numerical results are presented in Section~\ref{sec:numerical results}.  Finally, some concluding remarks are made in Section~\ref{sec:conclusion}.

	\section{The numerical scheme}
	\label{sec:numerical scheme}
	
\subsection{Fourier pseudo-spectral approximations}

The Fourier pseudo-spectral method is also referred as the Fourier collocation spectral method. It is closely related to the Fourier spectral method, but complements the basis by an additional pseudo-spectral basis, which allows to represent functions on a quadrature grid. This simplifies the evaluation of certain operators, and can considerably speed up the calculation when using fast algorithms such as the fast Fourier transform (FFT); see the related descriptions in~\cite{Boyd2001, cheng2015fourier, cheng16b, cheng16a, gottlieb12a, gottlieb12b, HGG2007, zhangC18a, zhangC17a}.

To simplify the notation in our pseudo-spectral analysis, we assume that the domain is given by $\Omega = (0,1)^3$, $N_x = N_y = N_z =: N\in\mathbb{N}$ and $N \cdot h = 1$. We further assume that $N$ is odd:
	\[
N = 2K+1, \quad \mbox{for some} \ K\in\mathbb{N}.
	\]
The analyses for more general cases are a bit more tedious, but can be carried out without essential difficulty. The spatial variables are evaluated on the standard 3D numerical grid $\Omega_N$, which is defined by grid points $(x_i, y_j, z_k)$, with $x_i = i h$, $y_j=jh$, $z_k = k h$, $0 \le i , j, k \le 2K +1$. This description for three-dimensional mesh ($d=3$) can here and elsewhere be trivially modified for the two-dimensional case ($d=2$).

We define the grid function space
	\begin{equation}
\mathcal{G}_N := \left\{ f:\mathbb{Z}^3 \to \mathbb{R} \ \middle| \ f \ \mbox{is $\Omega_N$-periodic} \right\} .
	\end{equation}
Given any periodic grid functions $f,g\in\mathcal{G}_N$, the $\ell^2$ inner product and norm are defined as
	\begin{eqnarray}
 \left\langle f , g \right\rangle  := h^3 \sum_{i,j,k=0}^{N -1}   f_{i,j,k}\cdot g_{i,j,k} , \quad \left\| f \right\|_2 := \sqrt{ \left\langle f , f \right\rangle } .
	\label{spectral-coll-inner product-1}
	\end{eqnarray}
The zero-mean grid function subspace is denoted $\mathring{\mathcal{G}}_N := \left\{ f\in \mathcal{G}_N \ \middle| \  \langle f, 1\rangle =:  \overline{f}  = 0\right\}$.  For  $f\in \mathcal{G}_N$, we have the discrete Fourier expansion
	\begin{equation}
f_{i,j,k} = \sum_{\ell,m,n=-K}^{K} \hat{f}_{\ell,m,n}^N \exp \left( 2 \pi {\rm i} ( \ell x_i + m y_j + n z_k ) \right) ,
	\label{spectral-coll-1}
	\end{equation}
where the discrete Fourier coefficients are given by 
	\begin{equation}
\hat{f}_{\ell,m,n}^N := h^3\sum_{i,j,k = 0}^{N-1} f_{i,j,k}\exp\left(-2\pi {\rm i} \left(\ell x_i + m x_j + n x_k \right)\right) . 
	\end{equation}
The collocation Fourier spectral first and second order \emph{derivatives} of $f$ are defined as
	\begin{eqnarray}
{\cal D}_x f_{i,j,k} := \sum_{\ell,m,n = -K}^{K}  \left( 2 \pi {\rm i} \ell \right) \hat{f}_{\ell,m,n}^N  \exp \left( 2 \pi {\rm i} ( \ell x_i + m y_j + n z_k ) \right) ,
	\label{spectral-coll-2-1}
	\\
{\cal D}_x^2 f_{i,j,k} := \sum_{\ell,m,n = -K}^{K}   \left( - 4 \pi^2 \ell^2 \right) \hat{f}_{\ell,m,n}^N \exp \left( 2 \pi {\rm i} ( \ell x_i + m y_j + n z_k) \right) .
	\label{spectral-coll-2-3}
	\end{eqnarray}
The differentiation operators in the $y$ and $z$ directions, ${\cal D}_y$, ${\cal D}_y^2$, ${\cal D}_z$ and ${\cal D}_z^2$ can be defined in the same fashion. In turn, the discrete Laplacian, gradient and divergence operators are given by
	\begin{equation}
\Delta_N f :=  \left( {\cal D}_x^2  + {\cal D}_y^2 + {\cal D}_z^2 \right) f , \quad \nabla_N f := \left(
	\begin{array}{c}
{\cal D}_x f
	\\
{\cal D}_y f
	\\
{\cal D}_z f
	\end{array}
\right),  \quad  \nabla_N \cdot \left(
	\begin{array}{c}
f _1
	\\
f _2
	\\
f_3
	\end{array}
\right)  := {\cal D}_x f_1 + {\cal D}_y f_2 + {\cal D}_z f_3 ,
	\label{spectral-coll-3}
	\end{equation}
at the point-wise level. It is straightforward to verify that
	\begin{equation}
\nabla_N \cdot \nabla_N f = \Delta_N f .
	\label{spectral-coll-4-a}
	\end{equation}
See the derivations in the related references~\cite{Boyd2001, canuto82,
Gottlieb1977}.

	\begin{defi}
Suppose that the grid function $f\in\mathcal{G}_N$ has the discrete Fourier expansion (\ref{spectral-coll-1}).  Its spectral extension into the trigonometric polynomial space $\mathcal{P}_K$ (the space of trigonometric polynomials of degree at most $K$) is defined as
	\begin{equation}
f_S (x,y,z)  = \sum_{\ell,m,n=-K}^{K}   \hat{f}_{\ell,m,n}^N \exp \left( 2 \pi {\rm i} ( \ell x + m y + n z) \right) .
	\label{spectral-coll-projection-2}
	\end{equation}
We write $S_N(f) = f_S$ and call $S_N:\mathcal{G}_N \to \mathcal{P}_K$ the spectral interpolation operator. Suppose $g\in C_{\rm per}(\Omega,\mathbb{R})$. We define the grid projection $Q_N: C_{\rm per}(\Omega,\mathbb{R})\to\mathcal{G}_N$ via
	\begin{equation}
Q_N(g)_{i,j,k} := g(x_i,y_j,z_k),
	\end{equation}
The resultant grid function may, of course, be expressed as a discrete Fourier expansion:
	\[
Q_N(g)_{i,j,k} = \sum_{\ell,m,n=-K}^{K}   \widehat{Q_N(g)}_{\ell,m,n}^N \exp \left( 2 \pi {\rm i} ( \ell x_i + m y_j + n z_k)\right) .
	\]
We define the de-aliasing operator $R_N :C_{\rm per}(\Omega,\mathbb{R}) \to \mathcal{P}_K$ via $R_N := S_N(Q_N)$. In other words,
	\begin{eqnarray}
R_N(g)(x,y,z)  = \sum_{\ell,m,n=-K}^{K}   \widehat{Q_N(g)}_{\ell,m,n}^N \exp \left( 2 \pi {\rm i} ( \ell x + m y + n z)\right) .
    \label{spectral-coll-projection-3}
	\end{eqnarray}
Finally, for any $g\in L^2(\Omega,\mathbb{R})$, we define the (standard) Fourier projection operator $P_N:L^2(\Omega,\mathbb{R}) \to {\mathcal P}_K$ via
	\[
P_N(g) (x,y,z) = \sum_{\ell,m,n = -K}^K \hat{g}_{\ell,m,n} \exp \left( 2 \pi {\rm i} ( \ell x + m y + n z)\right),
	\]
where
	\[
\hat{g}_{\ell,m,n} = \int_\Omega g(x,y,z) \exp\left(- 2\pi{\rm i} \left(\ell x+m y+n z \right)\right)  d\x,
	\]
are the (standard) Fourier coefficients.
	\end{defi}
	
	\begin{rem}
Note that, in general, for $g\in C_{\rm per}(\Omega,\mathbb{R})$, $P_N(g) \ne R_N(g)$, and, in particular,
	\[
\hat{g}_{\ell,m,n} \ne \widehat{Q_N(g)}_{\ell,m,n}^N .
	\]
However, if $g\in\mathcal{P}_K$ to begin with, then $\hat{g}_{\ell,m,n} = \widehat{Q_N(g)}_{\ell,m,n}^N$. In other words, $R_N :\mathcal{P}_K \to \mathcal{P}_K$ is the identity operator.
	\end{rem}

To overcome a key difficulty associated with the $H^m$ bound of the nonlinear term  obtained by collocation interpolation, the following lemma is introduced.
	\begin{lem}
	\label{lemma:aliasing error}
Suppose that $m$ and $K$ are non-negative integers, and,  as before, assume that $N = 2K+1$. For any $\varphi \in {\cal P}_{mK}$ in $\mathbb{R}^d$, we have the estimate
	\begin{equation}
\left\| R_N (\varphi) \right\|_{H^r}  \le  m^{\frac{d}{2}}  \left\|  \varphi \right\|_{H^r},
	\label{spectral-coll-projection-4}
	\end{equation}
for any non-negative integer $r$.
	\end{lem}
The case of $r=0$ was proven in Weinan E's earlier papers~\cite{E92, E93}.  The case of $r \ge 1$ was analyzed in a recent article by Gottlieb and Wang~\cite{gottlieb12b}. The proof of the following estimate can be found in~\cite{canuto82}.

	\begin{prop}
Suppose that $\phi \in H_{\rm per}^m(\Omega)$ and $m> \frac{d}{2}$. Then, there is some constant $C>0$, such that
	\begin{equation}
  \| \phi - R_N(\phi) \|_{H^k} \leq C \| \phi \|_{H^m} h^{m-k} , \quad \mbox{for} \quad  0 \le k \le m .  
   \label{spectral-approximation} 
	\end{equation} 
	\end{prop}
 
 We also note the following identity:
 
	\begin{prop}
Suppose that $\phi\in C_{\rm per}(\Omega)$. Then
	\begin{equation} 
{\cal D}_x Q_N(\phi) = Q_N(\partial_x R_N(\phi) ).
	\label{interpolationl-2}
	\end{equation}
Similar identities are available to the higher order derivatives.
	\end{prop}

In addition, we introduce the discrete  fractional operator $(-\Delta_N)^\gamma$ (with $\gamma >0$): 
	\begin{equation}
(-\Delta_N)^\gamma f_{i,j,k} := \sum_{\ell,m,n=-K}^{K} \left( 4 \pi^2 ( \ell^2 + m^2 + n^2 ) \right)^\gamma \hat{f}_{\ell,m,n}^N  \exp \left( 2 \pi {\rm i} ( \ell x_i + m y_j + n z_k) \right) ,
	\label{spectral-coll-4-b}
	\end{equation}
for a grid function $f$ with the discrete Fourier expansion as (\ref{spectral-coll-1}). Similarly, for a grid function $f\in\mathring{\mathcal{G}}_N$ of  (discrete) mean zero, a discrete version of the operator $(-\Delta)^{-\gamma}$ may be defined as
	\begin{equation}
(-\Delta_N)^{-\gamma} f_{i,j,k} := \sum_{\substack{\ell,m,n =-K \\ 
(\ell,m,n) \ne {\bf 0}}}^K \left( 4 \pi^2 ( \ell^2 + m^2 + n^2 ) \right)^{-\gamma} \hat{f}_{\ell,m,n}^N \exp \left( 2 \pi {\rm i} ( \ell x_i + m y_j + n z_k) \right).
	\label{spectral-coll-5}
	\end{equation}
Observe that, in this way of defining the inverse operator, the result is a periodic grid function of zero mean, \emph{i.e}, $(-\Delta_N)^{-\gamma} f\in\mathring{\mathcal{G}}_N$.

Detailed calculations show that the following summation-by-parts formulas
are valid (see the related discussions in~\cite{chen12, chen14,   gottlieb12a, gottlieb12b}): for any periodic grid functions $f,g\in\mathcal{G}_N$,
	\begin{equation}
\left\langle f ,  \Delta_N  g  \right\rangle  = - \left\langle \nabla_N f ,  \nabla_N g   \right\rangle  ,    \quad \left\langle f ,  \Delta_N^2  g  \right\rangle =  \left\langle \Delta_N f ,  \Delta_N g   \right\rangle ,   \quad \left\langle f ,  \Delta_N^3  g  \right\rangle =  - \left\langle \nabla_N \Delta_N f ,  \nabla_N \Delta_N g   \right\rangle . 
	\label{spectral-coll-inner product-3}
	\end{equation}

Since the SPFC equation (\ref{equation-SPFC}) is an $H^{-1}$ gradient flow, we need  a discrete version of the norm $\| \cdot \|_{H^{-1}}$ defined on $\mathring{\mathcal G}_N$. For any $f, g \in \mathring{\mathcal G}_N$, we define
\begin{eqnarray}
  \langle f,  g \rangle_{-1,N} := \left\langle f ,  ( - \Delta_N )^{-1} g  \right\rangle
  = \left\langle ( - \Delta_N )^{-\frac12}  f , ( - \Delta_N )^{-\frac12} g   \right\rangle, 
  \label{spectral-coll-inner product-6}
\end{eqnarray} 
so that the $\| \cdot \|_{-1,N}$ norm could be introduced as 
	\begin{equation}
\| f \|_{-1,N} := \sqrt{ \langle f , f \rangle_{-1,N} } = \|  ( - \Delta_N )^{-\frac12}  f  \|_2 .
	\label{spectral-coll-inner product-5}
	\end{equation}


In addition to the standard $\ell^2$ norm, we also introduce the $\ell^p$, $1\le p <\infty$, and $\ell^\infty$ norms for a grid function $f\in\mathcal{G}_N$:
\begin{equation}
 \nrm{f}_{\infty} := \max_{i,j,k} |f_{i,j,k}| ,   \qquad
 \nrm{f}_{p}  := \Bigl( h^3\sum_{i,j,k=0}^{N-1} |f_{i,j,k} |^p \Bigr)^{\frac{1}{p}} , \quad 1\leq p < \infty.  \label{spectral-defi-Lp}
\end{equation}
The discrete $H^1$ and $H^2$ norms are introduced as
\begin{equation} 
  \| f \|_{H_N^1}^2 = \| f \|_2^2 + \| \nabla_N f \|_2^2 ,  \quad 
  \| f \|_{H_N^2}^2 = \| f \|_{H_N^1}^2 + \| \Delta_N f \|_2^2 . 
   \label{spectral-defi-Hm}
\end{equation}
For any periodic grid function $\phi\in\mathcal{G}_N$, the discrete SPFC energy is defined as
	\be
E_N (\phi) := \frac14 \| \nabla_N \phi \|_4^4 + \frac{a}{2} \| \phi \|_2^2 - \| \nabla_N \phi \|_2^2 + \frac12 \nrm{\Delta_N \phi}_2^2 .
	\label{energy-discrete-spectral}
	\ee
	
The inequalities in the next lemma will play an important role in the convergence analysis. 

\begin{lem}  \label{lem: inequality} 
  For any $f \in  \mathcal{G}_N$, we have 
\begin{eqnarray} 
    \| \nabla_N f \|_2 \le \| f \|_2^{2/3} \cdot \| \nabla_N \Delta_N f \|_2^{1/3} ,   \quad 
   \| \Delta_N f \|_2 \le \| f \|_2^{1/3} \cdot \| \nabla_N \Delta_N f \|_2^{2/3} . 
   \label{inequality-0-1}    
\end{eqnarray} 
\end{lem}

\begin{proof} 
An application of the summation by parts formula~\eqref{spectral-coll-inner product-3} gives 
\begin{equation} 
  \| \nabla_N f \|_2^2 = - \langle f , \Delta_N f \rangle \le \| f \|_2 \cdot \| \Delta_N f \|_2 . \label{inequality-1} 
 \end{equation} 
 Meanwhile, another summation by parts formula reveals that 
 \begin{equation} 
  \| \Delta_N f \|_2^2 = - \langle \nabla_N f , \nabla_N \Delta_N f \rangle \le \| \nabla_N f \|_2 \cdot \| \nabla_N \Delta_N f \|_2 . \label{inequality-2} 
 \end{equation}  
 Therefore, a combination of~\eqref{inequality-1} and \eqref{inequality-2} leads to 
 \begin{align} 
   \| \nabla_N f \|  \le  \| f \|_2^{1/2} \cdot \| \Delta_N f \|_2^{1/2} 
     & \le  \| f \|_2^{1/2} \cdot ( \| \nabla_N f \|_2^{1/2} \cdot \| \nabla_N \Delta_N f \|_2^{1/2}  )^{1/2}
	\nonumber 
	\\
&= \| f \|_2^{1/2} \cdot \| \nabla_N f \|_2^{1/4} \cdot \| \nabla_N \Delta_N f \|_2^{1/4} ,    
	\label{inequality-3-1} 
	\end{align}  
which in turn results in 
	\begin{equation} 
\| \nabla_N f \|^{3/4} \le  \| f \|_2^{1/2} \cdot \| \nabla_N \Delta_N f \|_2^{1/4} ,   
   \quad \mbox{i.e.,}  \quad  \| \nabla_N f \| \le  \| f \|_2^{2/3} \cdot \| \nabla_N \Delta_N f \|_2^{1/3} . 
	\label{inequality-3-2} 
	\end{equation}     
 This finishes the proof of the first inequality in~\eqref{inequality-0-1}. For the second part, we make use of the preliminary estimate~\eqref{inequality-2}, and substitute the derived inequality in~\eqref{inequality-3-2}: 
 \begin{eqnarray} 
  \| \Delta_N f \|_2  \le \| \nabla_N f \|_2^{1/2} \cdot \| \nabla_N \Delta_N f \|_2^{1/2} 
 &\le& (\| f \|_2^{2/3} \cdot \| \nabla_N \Delta_N f \|_2^{1/3} ) ^{1/2} \cdot \| \nabla_N \Delta_N f \|_2^{1/2}  \nonumber 
\\
  &=& 
  \| f \|_2^{1/3} \cdot \| \nabla_N \Delta_N f \|_2^{2/3} , 
  \label{inequality-4} 
 \end{eqnarray}     
 so that the second inequality in~\eqref{inequality-0-1} is valid. This completes the proof for Lemma~\ref{lem: inequality}.  
	\end{proof} 
	
The following result corresponds to a discrete Sobolev embedding from $H_N^2$ to $W_N^{1,6}$ in the pseudo-spectral space. Similar discrete embedding estimates, in the lower order ones, could be found in Lemma 2.1 of~\cite{cheng16a}; also see the related results \cite{feng2017preconditioned,  fengW17c} in the finite difference version. A direct calculation is not able to derive these inequalities; instead, a discrete Fourier analysis has to be applied in the derivation; the details will be left in Appendix~\ref{proof:Prop 2.2}.

	\begin{prop}
\label{prop:embedding} 
  For any periodic grid function $f$, we have 
\begin{eqnarray} 
  \| \nabla_N f \|_6 \le C \| \Delta_N f \|_2 ,  \label{embedding-0} 
\end{eqnarray} 
for some constant $C$ only dependent on $\Omega$. 
	\end{prop}

	\section{Solvability and stability of the fully discrete schemes} 
	\label{sec:stability}

A modified second order BDF temporal discretization is applied to the SPFC equation, combined with Fourier pseudo-spectral approximation in space: 
	\begin{equation} 
\frac{\frac32 \phi^{k+1}- 2 \phi^{k}+ \frac12 \phi^{k-1}}{\dt}  = \Delta_N \mu_i^{k+1}, \quad i = 1,2.
	\end{equation}
The discrete chemical potential is chosen in two different ways:  scheme (1) ($-|\nabla\phi|^2$ is destabilizing)

	\begin{align}
\mu_1^{k+1} & =   - \nabla_N \cdot ( | \nabla_N \phi^{k+1} |^2 \nabla_N \phi^{k+1} )  + a \phi^{k+1}
	\nonumber 
	\\
& \quad  + 2 \Delta_N ( 2 \phi^k - \phi^{k-1} ) - A \dt \Delta_N (\phi^{k+1} - \phi^k) + \Delta_N^2 \phi^{k+1} ,
	\label{scheme-BDF-SPFC-1}
	\end{align}
and, scheme (2) ($-\frac{\varepsilon}{2}\phi^2$ is destabilizing)
	\begin{align} 
\mu_2^{k+1} & =  - \nabla_N \cdot ( | \nabla_N \phi^{k+1} |^2 \nabla_N \phi^{k+1} )  -\varepsilon (2\phi^k-\phi^{k-1})
	\nonumber 
	\\
& \quad - A \dt \Delta_N (\phi^{k+1} - \phi^k) + \left(1+\Delta_N\right)^2 \phi^{k+1}  .
	\label{scheme-BDF-SPFC-1-alt}
	\end{align}
Comparing with the standard BDF2 algorithm, the concave diffusion terms have been explicitly updated, both for the sake of unique solvability and stability. Furthermore, a second order Douglas-Dupont-type regularization term has to be added in the chemical potential to ensure energy stability. Similar ideas could be found  for the epitaxial thin film growth model~\cite{fengW17c} and the Cahn-Hilliard model~\cite{yan17}. On the other hand, a careful analysis will reveal a subtle difference in the artificial regularization between this work and the two earlier ones, as will be demonstrated in the next section.  

	\begin{rem}
We point out that the second scheme~\eqref{scheme-BDF-SPFC-1-alt} is inspired by the ideas in the paper~\cite{shin2016}, where they construct an energy stable numerical scheme for the original PFC equation based on a modified Crank-Nicolson time discretization and a new convex splitting, one that views $-\frac{\varepsilon}{2}\phi^2$ as the destabilizing (concave) term in the energy. They show in some numerical tests that this new approach yields more accurate numerical results compared to the scheme in~\cite{hu09}. The scheme here is new in that we use the BDF2 approach, and it is applied to a different energy/gradient flow than the one considered in~\cite{shin2016}.
	\end{rem}

We will always assume that
	\[
 \varepsilon <1 \quad \iff \quad  a := 1-\varepsilon >0 .	
	\]
By $\Phi$ we denote the exact PDE solution for~\eqref{equation-SPFC}, and the initial value for the numerical methods is taken as 
	\[
\phi^0  = Q_N\left(\Phi(\, \cdot \, , t=0)\right).
	\]
We observe that~\eqref{scheme-BDF-SPFC-1} and \eqref{scheme-BDF-SPFC-1-alt} are two-step numerical methods, so that a ``ghost" point extrapolation for $\phi^{-1}$ is useful. To preserve the second order accuracy in time, we apply the following approximation: 
	\begin{equation} 
  \phi^{-1} = \phi^0 - \dt \Delta_N \mu^0, \quad  \mu^0 := 
  - \nabla_N \cdot ( | \nabla_N \phi^0 |^2 \nabla_N \phi^0 ) 
   + a \phi^0 + 2 \Delta_N \phi^0 + \Delta_N^2 \phi^0 .  \label{scheme-BDF-SPFC-initial-1}
	\end{equation}
A careful Taylor expansion indicates an $O (\dt^2 + h^m)$ accuracy for such an approximation: 
\begin{eqnarray} 
  \| \phi^{-1} - \Phi^{-1} \|_2 \le C (\dt^2 + h^m). \label{scheme-BDF-SPFC-initial-2}
\end{eqnarray}  

	\begin{thm}
	\label{SPFC solvability} 
 Given $\phi^k, \phi^{k-1} \in \mathcal{G}_N$, with $\overline{\phi^k} = \overline{\phi^{k-1}}$, there exist unique solutions -- both labeled $\phi^{k+1} \in \mathcal{G}_N$, for simplicity,  though they are in general distinct -- for the numerical schemes~\eqref{scheme-BDF-SPFC-1} and \eqref{scheme-BDF-SPFC-1-alt}. The  schemes are mass conservative, i.e., $\overline{\phi^k} \equiv \overline{\phi^0} := \beta_0$, for any $k \ge 0$. 
	\end{thm}  
	
	\begin{proof} 
We prove the result only for~\eqref{scheme-BDF-SPFC-1}; the proof for the other is similar.  By taking a discrete summation of \eqref{scheme-BDF-SPFC-1}, and making use of the fact that $\overline{\Delta_N \mu_i^{k+1} } =0$, $i=1,2$, as well as the mass conservation of the previous time steps: $\overline{\phi^k} = \overline{\phi^{k-1}} = \beta_0$, we are able to conclude that $\overline{\phi^{k+1}} = \beta_0$, for any $k \ge 0$, provided a solution exists.

Next, observe that \eqref{scheme-BDF-SPFC-1} can be rewritten as 
	\begin{equation} 
\mathcal{N}_N [\phi] = f := - 2 \dt \Delta_N ( 2 \phi^k - \phi^{k-1} ) - A \dt^2 \Delta_N \phi^{k} , 
	\label{eqn:operator-1}
	\end{equation}
where
	\begin{align}
\mathcal{N}_N [\phi]  & := (-\Delta_N)^{-1}\left( \frac32 \phi - 2 \phi^{k} + \frac12\phi^{k-1}\right) - \dt \nabla_N \cdot ( | \nabla_N \phi |^2 \nabla_N \phi ) 
+ a \dt \phi 
	\nonumber 
	\\
& \quad  - A \dt^2 \Delta_N \phi + \dt \Delta_N^2 \phi . 
	\label{eqn:operator-2}  
	\end{align}
The nonlinear equation \eqref{eqn:operator-1} can be recast as a minimization problem for the following discrete energy functional: 
	\begin{align}
F_N [\phi] & :=  \frac{1}{3} \nrm{ \frac32\phi - 2 \phi^{k} + \frac12 \phi^{k-1}}_{-1,N}^2  + \frac{\dt}{4} \| \nabla_N \phi \|_{4}^4 + \frac{a \dt}{2} \| \phi \|_2^2 + \frac{A \dt^2}{2} \| \nabla_N \phi \|_{2}^2  
	\nonumber
	\\
& \quad + \frac{\dt}{2} \| \Delta_N \phi \|_2^2 - \langle f ,\phi \rangle ,
\label{eqn:min-energy-1}
	\end{align}
for any $\phi \in \mathcal{G}_N$. In turn, the strict convexity of $F_N$ (in terms of $\phi$), over the hyperplane of discrete functions satisfying the mass condition $\bar{\phi} = \beta_0$, implies a unique numerical solution for \eqref{scheme-BDF-SPFC-1}. 
	\end{proof}

	\begin{thm} \label{SPFC-energy stability}
For $k \ge 1$, define the discrete modified energies
	\begin{equation}
\mathcal{E}_{N,1} (\phi^{k+1}, \phi^k) := E_N ( \phi^{k+1})+\frac{1}{4\dt} \| \phi^{k+1}- \phi^k\|_{-1,N}^2 + \| \nabla_N ( \phi^{k+1} - \phi^k ) \|_2^2 ,  
	\label{discrete energy}
	\end{equation}
and
	\begin{equation}
\mathcal{E}_{N,2} (\phi^{k+1}, \phi^k) := E_N ( \phi^{k+1})+\frac{1}{4\dt} \| \phi^{k+1}- \phi^k\|_{-1,N}^2 + \frac{\varepsilon}{2} \|  \phi^{k+1} - \phi^k  \|_2^2 . 
	\label{discrete energy-alt}
	\end{equation}
Provided that 
	\begin{equation}
A \ge \frac{\varepsilon^2}{16},
	\label{stability-cond-1-2}
	\end{equation}
solutions of the numerical schemes~\eqref{scheme-BDF-SPFC-1} ($i=1$) and \eqref{scheme-BDF-SPFC-1-alt} ($i=2$) satisfy the respective dissipation properties
	\begin{equation}
\mathcal{E}_{N,i} ( \phi^{k+1}, \phi^k) 
\le \mathcal{E}_{N,i} ( \phi^k, \phi^{k-1}) , \quad i = 1, 2.
	\label{SPFC-eng stab-est}
	\end{equation}
	\end{thm}

	\begin{proof} 
  Since $\phi^{k+1} - \phi^k \in \mathring{\mathcal{G}}_N$, we take a discrete inner product of~\eqref{scheme-BDF-SPFC-1} with $(-\Delta_N)^{-1} (\phi^{k+1} - \phi^k)$. The following identities can be derived: 
	\begin{align} 
& \hspace{-0.3in} \frac{1}{\dt} \left\langle  \frac32 \phi^{k+1} - 2 \phi^k + \frac12 \phi^{k-1} , 
  (-\Delta_N)^{-1} (\phi^{k+1} - \phi^k)  \right\rangle
	\nonumber 
	\\
& = \frac{1}{\dt}  \left( \frac32 \|  \phi^{k+1} - \phi^k \|_{-1,N}^2 
  - \frac12 \langle \phi^k - \phi^{k-1} , \phi^{k+1} - \phi^k \rangle_{-1,N} \right) 
	\nonumber 
	\\
& = \frac{1}{\dt}  \left( \frac54 \| \phi^{k+1} - \phi^k \|_{-1,N}^2 - \frac14 \| \phi^k - \phi^{k-1} \|_{-1,N}^2  + \frac{1}{4} \nrm{\phi^{k+1}-2\phi^k+\phi^{k-1}}_{-1,N}^2\right) , 
    \label{scheme-BDF-stability-1} 
	\\
& \hspace{-0.3in} \left\langle  \Delta_N ( \nabla_N \cdot ( | \nabla_N \phi^{k+1} |^2 \nabla_N \phi^{k+1} )  )  , 
  (-\Delta_N)^{-1} (\phi^{k+1} - \phi^k)  \right\rangle  
	\nonumber 
	\\
& = \left\langle  - \nabla_N \cdot ( | \nabla_N \phi^{k+1} |^2 \nabla_N \phi^{k+1} )  ,  \phi^{k+1} - \phi^k  \right\rangle   
	\nonumber 
	\\
& = \left\langle  | \nabla_N \phi^{k+1} |^2 \nabla_N \phi^{k+1} )  ,  \nabla_N ( \phi^{k+1} - \phi^k  ) \right\rangle = \frac14 (  \| \nabla_N \phi^{k+1} \|_4^4 - \| \nabla_N \phi^k \|_4^4 ) +R_4 , 
    \label{scheme-BDF-stability-2} 
    \end{align}
where the non-negative remainder $R_4\ge 0$ depends upon $\phi^{k+1}$, $\phi^k$. Furthermore,
    \begin{align}
& \hspace{-0.3in} \left\langle  - \Delta_N \phi^{k+1}  , (-\Delta_N)^{-1} (\phi^{k+1} - \phi^k)  \right\rangle =  \langle  \phi^{k+1}  , \phi^{k+1} - \phi^k  \rangle  
	\nonumber 
	\\
& = \frac12 \left(  \| \phi^{k+1} \|_2^2 - \| \phi^k \|_2^2  + \| \phi^{k+1} - \phi^k \|_2^2  \right)  , 
    \label{scheme-BDF-stability-3} 
	\\
& \hspace{-0.3in} \left\langle - \Delta_N^3 \phi^{k+1}  , (-\Delta_N)^{-1} (\phi^{k+1} - \phi^k)  \right\rangle  =  \langle  \Delta_N \phi^{k+1}  , \Delta_N (\phi^{k+1} - \phi^k)  \rangle
	\nonumber 
	\\
& = \frac12 \left(  \| \Delta_N \phi^{k+1} \|_2^2 - \| \Delta_N \phi^k \|_2^2  + \| \Delta_N ( \phi^{k+1} - \phi^k ) \|_2^2  \right)  , 
    \label{scheme-BDF-stability-4}     
	\\
& \hspace{-0.3in} \dt \left\langle  \Delta_N^2 ( \phi^{k+1} - \phi^k ) ,  (-\Delta_N)^{-1} (\phi^{k+1} - \phi^k)  \right\rangle  =  \dt \| \nabla_N  ( \phi^{k+1} - \phi^k ) \|_2^2  , 
    \label{scheme-BDF-stability-5}  
	\\
& \hspace{-0.3in} 2  \left\langle  - \Delta_N^2  ( 2 \phi^k - \phi^{k-1})  , (-\Delta_N)^{-1} (\phi^{k+1} - \phi^k)  \right\rangle  = -2  \left\langle  \nabla_N ( 2 \phi^k - \phi^{k-1} ) , \nabla_N ( \phi^{k+1} - \phi^k)  \right\rangle
	\nonumber 
	\\
& = - \left(  \| \nabla_N \phi^{k+1} \|_2^2 - \| \nabla_N \phi^k \|_2^2 +\nrm{\nabla_N\left(\phi^{k+1}-\phi^k\right)}_2^2  \right) 
	\\
& \quad  + \nrm{\nabla_N\left(\phi^{k+1}-\phi^k \right)}_2^2 - \nrm{\nabla_N\left(\phi^k-\phi^{k-1} \right)}_2^2 + \nrm{\nabla_N\left(\phi^{k+1}-2\phi^k+\phi^{k-1}\right)}_2^2  . 
    \label{scheme-BDF-stability-6}        
	\end{align} 
Meanwhile, an application of Cauchy inequality indicates the following estimate: 
	\begin{equation} 
\frac{1}{\dt} \| \phi^{k+1} - \phi^k \|_{-1,N}^2 + A \dt \| \nabla_N ( \phi^{k+1} - \phi^k ) \|_2^2  \ge 2 A^{1/2}  \| \phi^{k+1} - \phi^k \|_2^2 . 
	\label{scheme-BDF-stability-7}    
	\end{equation} 
A combination of~\eqref{scheme-BDF-stability-1} -- \eqref{scheme-BDF-stability-6} and \eqref{scheme-BDF-stability-7} yields 
	\begin{equation} 
\mathcal{E}_N^{k+1} - \mathcal{E}_N^k - \| \nabla_N ( \phi^k - \phi^{k-1} ) \|_2^2 + \left(2 A^{1/2} + \frac{a}{2}\right)  \| \phi^{k+1} - \phi^k \|_2^2 + \frac12  \| \Delta_N ( \phi^{k+1} - \phi^k ) \|_2^2  \le 0 .   
	\label{scheme-BDF-stability-8}    
	\end{equation} 
In addition, under the condition that 
	\begin{equation} 
2 A^{1/2} + \frac{a}{2}   \ge \frac12 ,   \quad \mbox{or, equivalently,} \quad A \ge \frac{(1 - a)^2}{16} = \frac{\varepsilon^2}{16} ,  
	\label{condition-1} 
	\end{equation} 
we have 
	\begin{align}
\left( 2 A^{1/2} + \frac{\alpha}{2} \right) \| \phi^{k+1} - \phi^k \|_2^2  + \frac12  \| \Delta_N ( \phi^{k+1} - \phi^k ) \|^2  &  \ge  \frac12 \| \phi^{k+1} - \phi^k \|_2^2 + \frac12  \| \Delta_N ( \phi^{k+1} - \phi^k ) \|_2^2  
	\nonumber
	\\
& \ge \| \nabla_N ( \phi^{k+1} - \phi^k ) \|_2^2 .   
	\label{scheme-BDF-stability-9}    
	\end{align}  
Therefore, with an introduction of the modified discrete energy given by~\eqref{discrete energy}, we arrive at $\mathcal{E}_{N,1}^{k+1} \le  \mathcal{E}_{N,1}^k$, under the condition \eqref{stability-cond-1-2}.

For the second scheme~\eqref{scheme-BDF-SPFC-1-alt}, we need the additional identities
    \begin{align}
& \hspace{-0.3in} \left\langle - \Delta_N\left(1+\Delta_N\right)^2 \phi^{k+1}  , (-\Delta_N)^{-1} (\phi^{k+1} - \phi^k)  \right\rangle  =  \langle  \left(1+\Delta_N\right) \phi^{k+1}  , \left(1+\Delta_N\right) (\phi^{k+1} - \phi^k)  \rangle
	\nonumber 
	\\
& = \frac12 \left(  \| (1+\Delta_N) \phi^{k+1} \|_2^2 - \| (1+\Delta_N) \phi^k \|_2^2  + \| (1+\Delta_N) ( \phi^{k+1} - \phi^k ) \|_2^2  \right)  , 
    \label{scheme-BDF-stability-4-alt}
	\\
& \hspace{-0.3in}   \left\langle   \Delta_N  ( 2 \phi^k - \phi^{k-1})  , (-\Delta_N)^{-1} (\phi^{k+1} - \phi^k)  \right\rangle  = -  \left\langle   2 \phi^k - \phi^{k-1}  ,   \phi^{k+1} - \phi^k  \right\rangle
	\nonumber 
	\\
& = - \frac{1}{2}\left(  \| \phi^{k+1} \|_2^2 - \|  \phi^k \|_2^2 +\nrm{ \phi^{k+1}-\phi^k }_2^2  \right) 
	\\
& \quad  + \frac{1}{2}\nrm{\phi^{k+1}-\phi^k}_2^2 - \frac{1}{2}\nrm{ \phi^k-\phi^{k-1} }_2^2 + \frac{1}{2}\nrm{ \phi^{k+1}-2\phi^k+\phi^{k-1} }_2^2 . 
    \label{scheme-BDF-stability-6-alt}        
	\end{align} 
A combination of~\eqref{scheme-BDF-stability-1}, \eqref{scheme-BDF-stability-2}, \eqref{scheme-BDF-stability-4-alt} -- \eqref{scheme-BDF-stability-6-alt}  and \eqref{scheme-BDF-stability-7} yields 
	\begin{align} 
\mathcal{E}_{N,2}^{k+1} - \mathcal{E}_{N,2}^k +\left(2 A^{1/2}-\frac{\varepsilon}{2}\right) \|  \phi^k - \phi^{k-1}  \|_2^2  \le 0 .   
	\label{scheme-BDF-stability-8-alt}    
	\end{align} 
Thus, under the condition $ A \ge \frac{\varepsilon^2}{16}$, which is the same as for the first scheme,  we have $\mathcal{E}_{N,2}^{k+1} \le  \mathcal{E}_{N,2}^k$.
	\end{proof}

	\begin{rem}
We can also get a stability involving the chemical potential, if it is needed. It is observed that
	\begin{align*}
\dt\nrm{\nabla\mu_i^{k+1}}_2^2 & = \frac{1}{\dt}\nrm{\frac32 \phi^{k+1} - 2 \phi^k + \frac12 \phi^{k-1}}_{-1,N}^2
	\\
& = \frac{1}{\dt}\nrm{ \left( \phi^{k+1} - \phi^k \right) +\frac{1}{2} \left(  \phi^{k+1} - 2 \phi^k +  \phi^{k-1}\right) }_{-1,N}^2
	\\
& \le \frac{2}{\dt}\nrm{\phi^{k+1} - \phi^k }_{-1,N}^2 + \frac{1}{2\dt}\nrm{ \phi^{k+1} - 2 \phi^k +  \phi^{k-1}}_{-1,N}^2.
	\end{align*}
Thus, for example,
	\[
\frac{\dt}{2}\nrm{\nabla\mu_i^{k+1}}_2^2 \le \frac{1}{\dt}\nrm{\phi^{k+1} - \phi^k }_{-1,N}^2 + \frac{1}{4\dt}\nrm{ \phi^{k+1} - 2 \phi^k +  \phi^{k-1}}_{-1,N}^2 .
	\]
The first term on the right-hand side was completely utilized in our stability analyses. But, at the expense of a larger splitting parameter $A$, some of the term can be spared, and a norm stability of the chemical potential can be derived.	
	\end{rem}

	\begin{rem}
The energy stability of numerical methods for  gradient flow PDE have attracted a lot of attentions over the years. For the second order numerical scheme using the BDF temporal stencil, an artificial Douglas-Dupont regularization term has to be added to ensure the energy stability, as demonstrated in recent works~\cite{fengW17c, LiW18, yan17}, for the epitaxial thin film growth and Cahn-Hilliard equations, respectively. In particular, a careful comparison reveals that, the energy growth term coming from the concave diffusion process has to be compensated by the artificial diffusion term, as well as the additional stability in the temporal stencil. This in turn requires an artificial diffusion term to have the same diffusion power as the surface diffusion term, which becomes acceptable in the practical computations. 

For our first scheme~\eqref{scheme-BDF-SPFC-1}, on the other hand, if we carry out the same analysis for the SPFC equation~\eqref{equation-SPFC}, an artificial diffusion term with the power index as $- A \dt \Delta_N^3 (\phi^{k+1} - \phi^k)$, is needed to balance the energy growth part coming from the concave diffusion. This artificial diffusion term has an even higher diffusion power than the surface diffusion term, which in turn may lead to an extra numerical dissipation in the long time simulation. To avoid such a numerical artifact, we make use of the additional stability estimate coming from the surface diffusion part, so that only an artificial diffusion power index as $- A \dt \Delta_N (\phi^{k+1} - \phi^k)$ is needed in the numerical algorithm, and the energy stability could be derived using a more involved analysis. This choice avoids a higher order artificial diffusion term, therefore a reduced numerical dissipation is expected for the numerical effect. 
    \end{rem}

As a direct consequence of the energy stability, a uniform in time $H_N^2$ bound for the numerical solution is given as follows. 

	\begin{cor} \label{SPFC: H^2 bound} 
Suppose that the initial data are sufficiently regular so that for the first scheme~\eqref{scheme-BDF-SPFC-1}
	\[
E_N (\phi^0) + \frac{\dt}{4} \| \nabla_N \mu_N^0 \|_2^2 + \dt^2 \| \nabla_N \Delta_N \mu_N^0 \|_2^2 \le \tilde{C}_0,
	\]
and for the second scheme~\eqref{scheme-BDF-SPFC-1-alt},	
	\[
E_N (\phi^0) + \frac{\dt}{4} \| \nabla_N \mu_N^0 \|_2^2 + \frac{\varepsilon\dt^2}{2} \|  \Delta_N \mu_N^0 \|_2^2 \le \tilde{C}_0,
	\]
for some $\tilde{C}_0$ that is independent of $h$, and $A\ge \frac{\varepsilon^2}{16}$. Then we have the following uniform (in time) $H_h^2$ bound for the numerical solution: 
	\begin{equation}
\nrm{ \phi^m}_{H_N^2} \le \tilde{C}_1 ,  \quad \forall \, m \ge 1 , \label{SPFC-H2 stab-0}
	\end{equation} 
where $\tilde{C}_1>0$  depends on $\Omega$ and $\tilde{C}_0$, but is independent of $h$, $\dt$ and final time.  
	\end{cor}
	
	\begin{proof} 
As a result of \eqref{SPFC-eng stab-est}, for the first scheme~\eqref{scheme-BDF-SPFC-1}, the following energy bound is available: 
	\begin{align} 
E_N (\phi^m) & \le \mathcal{E}_{N,1} (\phi^{m}, \phi^{m-1}) \le   \mathcal{E}_{N,1} (\phi^0, \phi^{-1}) 
	\nonumber
	\\
& = E_N (\phi^0) + \frac{1}{4 \dt}  \| \phi^0 - \phi^{-1} \|_{-1,N}^2 + \| \nabla_N ( \phi^0 - \phi^{-1} ) \|_2^2   
	\nonumber 
	\\
&= E_N (\phi^0) + \frac{\dt}{4} \| \nabla_N \mu_N^0 \|_2^2 + \dt^2 \| \nabla_N \Delta_N \mu_N^0 \|_2^2 \le \tilde{C}_0 , \quad \forall m \ge 1. 
	\label{SPFC-H2 bound-1} 
	\end{align}
For the second scheme~\eqref{scheme-BDF-SPFC-1-alt}
	\begin{align} 
E_N (\phi^m) & \le \mathcal{E}_{N,2} (\phi^{m}, \phi^{m-1}) \le   \mathcal{E}_{N,2} (\phi^0, \phi^{-1}) 
	\nonumber
	\\
& = E_N (\phi^0) + \frac{1}{4 \dt}  \| \phi^0 - \phi^{-1} \|_{-1,N}^2 + \frac{\varepsilon}{2} \|  ( \phi^0 - \phi^{-1} ) \|_2^2   
	\nonumber 
	\\
&= E_N (\phi^0) + \frac{\dt}{4} \| \nabla_N \mu_N^0 \|_2^2 + \frac{\varepsilon\dt^2}{2} \|  \Delta_N \mu_N^0 \|_2^2 \le \tilde{C}_0 , \quad \forall m \ge 1. 
	\label{SPFC-H2 bound-1-alt} 
	\end{align}

On the other hand, the point-wise quadratic inequality, $\frac14 | \nabla_N \phi |^4 - 2 | \nabla_N \phi |^2 \ge - 4$, implies that 
\begin{eqnarray} 
  \frac14 \| \nabla_N \phi^m \|_4^4 - 2 \| \nabla_N \phi^m \|_2^2 \ge - 4 | \Omega | . 
  \label{SPFC-H2 bound-2} 
\end{eqnarray}
Its substitution into either~\eqref{SPFC-H2 bound-1} or \eqref{SPFC-H2 bound-1-alt} yields 
	\[ 
\frac{a}{2} \| \phi^m \|_2^2 + \| \nabla_N \phi^m \|_2^2 + \frac12  \| \Delta_N \phi^m \|_2^2  \le \tilde{C}_0 + 4 | \Omega | , 
	\]
so that, since $a > 0$,
	\[
\| \phi^m \|_{H_N^2}^2  \le \| \phi^m \|_2^2 + \| \nabla_N \phi^m \|_2^2 
  +  \| \Delta_N \phi^m \|_2^2 
  \le \frac{2}{a} (\tilde{C}_0 + 4 | \Omega |) , 
	\]
which implies 
	\begin{equation}
\| \phi^m \|_{H_N^2}. \le \Bigl( \frac{2}{a} (\tilde{C}_0 + 4 | \Omega |)  \Bigr)^{1/2} , \quad \forall m \ge 1 . 
	\label{SPFC-H2 bound-3} 
	\end{equation}
This completes the proof of Corollary~\ref{SPFC: H^2 bound}. 
	\end{proof} 

	\begin{rem}  
	\label{rem:W16 est}
As a combination of the uniform in time $H_N^2$ bound~\eqref{SPFC-H2 stab-0} and the discrete Sobolev embedding inequality~\eqref{embedding-0}, we arrive at a uniform in time $W_N^{1,6}$ estimate for the numerical solution: 
	\begin{equation} 
\| \nabla_N \phi^m \|_6 \le C \tilde{C}_1 ,  \quad \forall \ m \ge 1 .  
	\label{SPFC-W16 est-0} 
	\end{equation} 
This estimate will be useful in the convergence analysis presented below.  
\end{rem}

	\section{Convergence analysis}  \label{sec:convergence}

Now we proceed into the $\ell^\infty (0,T; \ell^2) \cap \ell^2 (0, T; H_N^3)$ convergence analysis for the proposed numerical scheme. With an initial data with sufficient regularity, we could assume that the exact solution has regularity of class $\mathcal{R}$: 
	\begin{equation}
\Phi \in \mathcal{R} := H^3 (0,T; C^0) \cap H^2 (0,T; H^4) \cap L^\infty (0,T; H^{m+6}).
	\label{assumption:regularity.1}
	\end{equation} 
We prove convergence only for the first scheme~\eqref{scheme-BDF-SPFC-1}; the details of the proof for the second are similar.
		
	\begin{thm}
	\label{thm:convergence}
Given initial data $\Phi_0 \in H_{\rm per}^{m+6} (\Omega)$, suppose the exact solution for SPFC equation~\eqref{equation-SPFC} is of regularity class $\mathcal{R}$. Then, provided $A \ge \frac{\varepsilon^2}{16}$ and $\dt$ and $h$ are sufficiently small, we have
	\begin{equation}
\max_{0\le n\le M}\| \Phi^n - \phi^n \|_2 +  ( \dt   \sum_{m=1}^{M} \| \nabla_N \Delta_N ( \Phi^m - \phi^m ) \|_2^2 )^{1/2}  \le C ( \dt^2 + h^m ),
	\label{SPFC-convergence-0}
	\end{equation}
where $C>0$ is independent of $\dt$ and $h$, and $\dt = T/M$.
	\end{thm}
	\begin{proof} 
For $\Phi \in \mathcal{R}$, a careful consistency analysis indicates the following truncation error estimate: 
	\begin{align}
\frac{\frac32 \Phi^{k+1} - 2 \Phi^k + \frac12 \Phi^{k-1}}{\dt}  &= \Delta_N 
   \Bigl(  - \nabla_N \cdot ( | \nabla_N \Phi^{k+1} |^2 \nabla_N \Phi^{k+1} ) + a \Phi^{k+1}
	\nonumber 
	\\
& \quad + 2 \Delta_N ( 2 \Phi^k - \Phi^{k-1} ) - A \dt \Delta_N (\Phi^{k+1} - \Phi^k) + \Delta_N^2 \Phi^{k+1} \Bigr) + \tau^{k+1} ,  
	\label{SPFC-consistency-1}
	\end{align} 
with $\| \tau^{k+1} \|_2 \le C (\dt^2 + h^m)$. The derivation of~\eqref{SPFC-consistency-1} is accomplished with the help of the spectral approximation estimate \eqref{spectral-approximation} and other related estimates; the details are left to interested readers. 

The numerical error function is defined at a point-wise level: 
	\begin{eqnarray} 
e^k := \Phi^k - \phi^k ,  \quad \forall k \ge 0 .   
	\label{SPFC-error function-1}
	\end{eqnarray} 
In turn, subtracting the numerical scheme \eqref{scheme-BDF-SPFC-1} from \eqref{SPFC-consistency-1} gives 
	\begin{align}
\frac{\frac32 e^{k+1} - 2 e^k + \frac12 e^{k-1}}{\dt}  &=  \Delta_N \Bigl(  - \nabla_N \cdot {\cal N} (\Phi^{k+1}, \phi^{k+1}) + a e^{k+1} + 2 \Delta_N ( 2 e^k - e^{k-1} )  
	\nonumber 
	\\
& \quad - A \dt \Delta_N ( e^{k+1} - e^k)  + \Delta_N^2 e^{k+1} \Bigr) + \tau^{k+1} ,  
	\label{SPFC-consistency-2-1} 
	\end{align}
where
	\begin{equation}
{\cal N} (\Phi^{k+1}, \phi^{k+1})  :=  | \nabla_N \phi^{k+1} |^2 \nabla_N e^{k+1} + ( \nabla_N ( \Phi^{k+1} + \phi^{k+1} ) \cdot \nabla_N e^{k+1} )  \nabla_N \Phi^{k+1} .  
	\label{SPFC-consistency-2-2}     
	\end{equation} 
	
Taking a discrete inner product of \eqref{SPFC-consistency-2-1}-\eqref{SPFC-consistency-2-2} with $e^{k+1}$, with a repeated application of summation by parts, we get 
	\begin{align} 
\left\langle \frac32 e^{k+1} - 2 e^k + \frac12 e^{k-1} , e^{k+1} \right\rangle  &  + \dt \| \nabla_N \Delta_N e^{k+1}  \|_2^2 +   A \dt^2 \langle \Delta_N ( e^{k+1} - e^k ) , \Delta_N e^{k+1} \rangle +  a \dt \| e^{k+1} \|_2^2 
	\nonumber
	\\
&  \le \dt \langle {\cal N} (\Phi^{k+1}, \phi^{k+1}) ,  \nabla_N \Delta_N e^{k+1}  \rangle + 2 \dt \langle \Delta_N ( 2 e^k - e^{k-1} )  ,  \Delta_N e^{k+1}   \rangle 
	\nonumber 
	\\
& \quad + \dt \langle e^{k+1} , \tau^{k+1} \rangle .
	\label{SPFC-convergence-1} 
	\end{align} 
The temporal stencil term could be analyzed in a standard way:  
\begin{eqnarray} 
  \hspace{-0.3in}
   \left\langle \frac32 e^{k+1} - 2 e^k + \frac12 e^{k-1} , e^{k+1} \right\rangle 
   &=& \frac14 \Bigl( \| e^{k+1} \|_2^2 - \| e^k \|_2^2 + \| 2 e^{k+1} - e^k \|_2^2 
   - \| 2 e^k - e^{k-1} \|_2^2   \nonumber 
 \\
   && 
   + \| e^{k+1} - 2 e^k + e^{k-1} \|_2^2 \Bigr) . 
   \label{SPFC-convergence-2}
\end{eqnarray} 
The artificial diffusion on the left hand side of~\eqref{SPFC-convergence-1} could be handled as follows:
	\begin{equation} 
\left\langle( \Delta_N ( e^{k+1} - e^k ) , \Delta_N e^{k+1}  \right\rangle_2  \ge \frac12 \left( \| \Delta_N e^{k+1} \|_2^2  - \| \Delta_N e^k \|_2^2 \right) .  
	\label{SPFC-convergence-3}
	\end{equation} 
The term associated with the local truncation error could be bounded with the help of Cauchy inequality: 
\begin{eqnarray} 
  \langle e^{k+1} , \tau^{k+1} \rangle  \le  \| e^{k+1}\|_2 \cdot \|  \tau^{k+1} \|_2  
  \le  \frac12 ( \| e^{k+1}\|_2^2 + \|  \tau^{k+1} \|_2^2 ) .  \label{SPFC-convergence-4}
\end{eqnarray} 

The term associated with the concave diffusion could be bounded in a similar way: 
\begin{eqnarray} 
  2 \langle \Delta_N ( 2 e^k - e^{k-1} )  ,  \Delta_N e^{k+1}   \rangle   
  \le 3 \|  \Delta_N e^{k+1} \|_2^2  + 2 \| \Delta_N e^k \|_2^2 
  + \| \Delta_N e^{k-1} \|_2^2 .  \label{SPFC-convergence-4-2}
\end{eqnarray} 
On the other hand, an application of the preliminary inequality~\eqref{inequality-0-1} reveals that
\begin{eqnarray} 
\| \Delta_N e^\ell \|_2^2 \le \| e^\ell \|_2^{2/3} \cdot \| \nabla_N \Delta_N e^\ell \|_2^{4/3} \le C \| e^\ell \|_2^2 + \frac{1}{24} \| \nabla_N \Delta_N e^\ell \|_2^2 , 
   \label{SPFC-convergence-4-3}    
\end{eqnarray} 
for $\ell=k-1, k, k+1$, in which the Young's inequality has been applied at the second step. Its substitution into~\eqref{SPFC-convergence-4-2} yields 
\begin{eqnarray} 
  2 \langle \Delta_N ( 2 e^k - e^{k-1} )  ,  \Delta_N e^{k+1}   \rangle   
  &\le&  \frac18 \|  \nabla_N \Delta_N e^{k+1} \|_2^2  + \frac{1}{12} \| \nabla_N \Delta_N e^k \|_2^2 + \frac{1}{24} \| \nabla_N \Delta_N e^{k-1} \|_2^2 \nonumber 
\\
  &&
  + C ( \| e^{k+1} \|_2^2 + \| e^k \|_2^2 + \| e^{k-1} \|_2^2 )  . 
   \label{SPFC-convergence-4-4}
\end{eqnarray}

For the nonlinear error term, we begin with the first term. An application of discrete H\"older inequality indicates that  
\begin{eqnarray} 
  \| | \nabla_N \phi^{k+1} |^2 \nabla_N e^{k+1}  \|_2 
  \le \|  \nabla_N \phi^{k+1}  \|_6^2 \cdot \| \nabla_N e^{k+1}  \|_6 
  \le  C \tilde{C}_1^2 \| \nabla_N e^{k+1}  \|_6  
  \le C \tilde{C}_1^2 \| \Delta_N e^{k+1}  \|_2 , 
   \label{SPFC-convergence-5-1}
\end{eqnarray} 
in which the preliminary estimate~\eqref{SPFC-W16 est-0} (which comes from the uniform in time $H_N^2$ bound~\eqref{SPFC-H2 stab-0}) and~\eqref{inequality-0-1} have been applied in the second and third steps, respectively. A bound for the second nonlinear error term could be derived in a similar way: 
\begin{eqnarray} 
  \| ( \nabla_N ( \Phi^{k+1} + \phi^{k+1} ) \cdot \nabla_N e^{k+1} )  \nabla_N \Phi^{k+1}  \|_2  \le C ( \tilde{C}_1^2 + (C^*)^2 ) \| \Delta_N e^{k+1}  \|_2 , 
   \label{SPFC-convergence-5-2}
\end{eqnarray} 
in which the constant $C^*$ corresponds to the preliminary estimate, $\|  \nabla_N \Phi^{k+1} \|_6 \le C^*$, for the exact solution $\Phi$. As a result, we get 
\begin{eqnarray} 
  \| {\cal N} (\Phi^{k+1} , \phi^{k+1} ) \|_2   
   \le C ( \tilde{C}_1^2 + (C^*)^2 ) \| \Delta_N e^{k+1}  \|_2   
   \le \tilde{C}_2 \| e^{k+1}  \|_2^{1/3} \cdot \| \nabla_N \Delta_N e^{k+1} \|_2^{2/3}  ,  \label{SPFC-convergence-5-3}
\end{eqnarray}   
with $\tilde{C}_2 = C ( (C^*)^2 + \tilde{C}_1^2 )$, in which the preliminary inequality~\eqref{inequality-0-1} has been applied in the second step. In turn, a bound for the nonlinear error inner product term becomes available
\begin{eqnarray} 
  \langle {\cal N} (\Phi^{k+1}, \phi^{k+1}) ,  \nabla_N \Delta_N e^{k+1}  \rangle  
  &\le& \| {\cal N} (\Phi^{k+1} , \phi^{k+1} ) \|_2  \cdot \| \nabla_N \Delta_N e^{k+1} \|_2  \nonumber 
\\
  &\le&  \tilde{C}_2  
  \| e^{k+1} \|_2^{1/3} \cdot   
  \| \nabla_N \Delta_N e^{k+1} \|_2^{5/3} )  \nonumber 
\\
  &\le& 
  \tilde{C}_{3} \| e^{k+1} \|_2^2   
  + \frac14 \| \nabla_N \Delta_N e^{k+1} \|_2^2  ,  
  \label{SPFC-convergence-5-4}
\end{eqnarray} 
with the Young's inequality applied in the last step. 

Subsequently, a substitution of~\eqref{SPFC-convergence-2}-\eqref{SPFC-convergence-4}, \eqref{SPFC-convergence-4-4} and \eqref{SPFC-convergence-5-4} into \eqref{SPFC-convergence-1} results in  
	\begin{align}
{\cal H}^{k+1} - {\cal H}^k & + \dt \| \nabla_N \Delta_N e^{k+1} \|_2^2    
	\nonumber 
	\\
& \le( 2 \tilde{C}_3 + C + 1 ) \dt \| e^{k+1} \|_2^2  + C \dt (  \| e^k \|_2^2 + \| e^{k-1} \|_2^2 )  + \dt \| \tau^{k+1} \|_2^2 ,  
	\label{SPFC-convergence-6-1} 
	\end{align}
where
	\begin{equation}
{\cal H}^{k+1} := \frac12 ( \| e^{k+1} \|_2^2 + \| 2 e^{k+1} - e^k \|_2^2 )  + A \dt^2 \| \Delta_N e^{k+1} \|_2^2 . 
    \label{SPFC-convergence-6-2} 
	\end{equation}  
Therefore, with an application of discrete Gronwall inequality, and making use of the fact that $\| \tau^{k+1} \|_2 \le C (\dt^2 + h^m)$, we arrive at 
\begin{eqnarray} 
   {\cal H}^{k+1}  + \dt \sum_{i=1}^{k+1} \| \nabla_N \Delta_N e^i \|_2^2  \le \hat{C} ( \dt^4 + h^{2m}) ,   \label{SPFC-convergence-6-3} 
\end{eqnarray}  
with $\hat{C}$ independent on $\dt$ and $h$. In turn, the desired convergence estimate is available 
\begin{eqnarray} 
   \| e^{k+1} \|_2  + \Bigl( \dt \sum_{i=1}^{k+1} \| \nabla_N \Delta_N e^i \|_2^2  \Bigr)^{1/2} \le C \hat{C}^{1/2} ( \dt^2 + h^m) ,   \label{SPFC-convergence-6-4} 
\end{eqnarray}  
This completes the proof of Theorem~\ref{thm:convergence}. 
	\end{proof}

	\section{Preconditioned steepest descent solver}
	\label{sec:PSD}
In this section we describe a preconditioned steepest descent (PSD) algorithm following the practical and  theoretical framework in~\cite{feng2017preconditioned}. We give the details for the first proposed BDF2 scheme~\eqref{scheme-BDF-SPFC-1}; the details for the second~\eqref{scheme-BDF-SPFC-1-alt} and related schemes will be similar. We first note that~\eqref{scheme-BDF-SPFC-1} can be recast as a minimization problem for~\eqref{eqn:min-energy-1}. One observes that the fully discrete scheme \eqref{scheme-BDF-SPFC-1} is the discrete variation of the strictly convex energy \eqref{eqn:min-energy-1} set equal to zero. The nonlinear scheme at a fixed time level may be expressed as~\eqref{eqn:operator-1}-\eqref{eqn:operator-2}. 

The essential idea of the PSD solver is to use a linearized version of the nonlinear operator as a pre-conditioner, or in other words, as a metric for choosing the search direction. In the following, we use the notation introduced in Theorem~\ref{SPFC solvability}. A linearized version of the nonlinear operator $\mathcal{N}_N$, denoted as $\mathcal{L}_N: \mathring{\mathcal{G}}_N \to \mathring{\mathcal{G}}_N$,  is defined as follows: 	\begin{equation*}
{\mathcal L}_N [\psi] :=  \frac{3}{2} (-\Delta_N)^{-1} \psi   - \dt \Delta_N \psi + a \dt \psi  
  - A \dt^2 \Delta_N \psi + \dt \Delta_N^2 \psi .
	\end{equation*}
In turn, this positive, symmetric operator could be used as a pre-conditioner for the numerical iteration. Specifically, this ``metric" is used to find an appropriate search direction for the steepest descent solver~\cite{feng2017preconditioned}. Given the current iterate $\phi_n\in \mathcal{G}_N$, we define the following \emph{search direction} problem: find $d_n \in \mathcal{G}_N$ such that
\[
{\mathcal L}_N [d_n] = r_n-\overline{r_n} , \quad r_n := f-\mathcal{N}_N [\phi_n],
\]
where $r_n$ is the nonlinear residual of the $n^{\rm th}$ iterate $\phi_n$. Of course, this equation can be efficiently solved using the Fast Fourier Transform (FFT).

Subsequently, the next iterate is obtained as
	\begin{equation}
\phi_{n+1} := \phi_n + \alpha_n d_n,
	\end{equation}
where $\alpha_n\in\mathbb{R}$ is the unique solution to the steepest descent line minimization problem
	\begin{equation}
\alpha_n := \operatorname*{argmin}_{\alpha\in\mathbb{R}} F_N [\phi_n + \alpha d_n]= \operatorname*{argzero}_{\alpha\in\mathbb{R}}\delta F_N [\phi_n + \alpha d_n](d_n) .
	\label{eqn-search}
	\end{equation}
Following similar techniques reported in~\cite{feng2017preconditioned} for the finite difference numerical method, a theoretical analysis ensures a geometric convergence of the iteration sequence: 
	\begin{equation} 
\| \phi_n - \phi^{k+1} \|_{H_N^2} \le \beta^n \| \phi_0 - \phi^{k+1} \|_{H_N^2} , \quad 0 < \beta< 1 , 
	\end{equation} 
where $\beta$ is independent of $N$ and $\phi^{k+1}$ stands for the exact numerical solution to \eqref{scheme-BDF-SPFC-1} at time level $k+1$, \emph{i.e.},  $\mathcal{N}_N [\phi^{k+1}] = f$.

	\begin{rem}
We note that our PSD method can be viewed as a quasi-Newton method, with an orthogonalization (line search) step. Indeed, $\mathcal{L}_N$ may be viewed as an approximation of the Jacobian. To fit more neatly into the framework of a traditional quasi-Newton method, one could just take step size equal to 1, so that the correction is just $\phi_{n+1} := \phi_n +  d_n$. Alternatively, one can just use quadratic line search methods to obtain an approximation of $\alpha_n$, call it $\alpha_n^{\rm q}$, to obtain a good-enough approximation that one can still prove a geometric convergence rate that is independent of $N$.
	\end{rem}

	\begin{rem} 
\textcolor{black}{There have been quite a few existing works~\cite{fengW17c, qiao17, shen12, wang10a} on the the slope-selection (SS) model for epitaxial thin film growth, including both numerical analysis and implementation issues. These existing works only deal with the 4-Laplacian term in an $L^2$ gradient flow. This work is, to our knowledge, the first attempt to address the energy stability for a 4-Laplacian energy term in an $H^{-1}$ gradient flow. Because of its highly nonlinear nature, a 4-Laplacian term in an $H^{-1}$ flow turns out to be much more challenging than in an $L^2$ flow, at both the theoretical and numerical levels. Furthermore, most existing works for the 4-Laplacian problem have been focused on either the finite difference and finite element spatial approximations, while this work reports a Fourier pseudo-spectral approach for the first time. Such a spectral approach has the advantage of avoiding complicated staggered mesh points in the evaluation of the gradient variables, which may lead to much reduced numerical accuracy for a complicated physical energy (see a related report~\cite{fengW17b} for the functionalized Cahn-Hilliard model). Meanwhile, the analysis for such a global spatial discretization becomes much more involved, as described in this work.} 
	\end{rem} 

	\section{Numerical results}
	\label{sec:numerical results}

	\subsection{Convergence test for the numerical scheme}

In this subsection we perform some numerical experiments to verify the accuracy order of the proposed numerical scheme. In particular, it is observed that the search direction and Poisson-like equations can also be efficiently solved by using the Fourier pseudo-spectral method (see the related discussions in~\cite{Boyd2001, cheng2015fourier, Gottlieb1977, HGG2007}) and Fast Fourier Transform (FFT). 

To test the convergence rate, we choose the following exact solution for \eqref{equation-SPFC} on the square domain $\Omega=(0, 1)^2$:
	\begin{equation}
	\label{eqn:init1}
\phi_e(x,y,t) = \frac{1}{2\pi}\sin (2 \pi x) \cos (2 \pi y) \cos(t).
	\end{equation}
We take $a =0.975$, and we choose the artificial diffusion coefficient as $A = 0.25$. The final time is taken as $T=0.16$.
        	
	\begin{figure}
	\centering
    \hbox{ 
	\includegraphics[width=3in,height=2.8in]{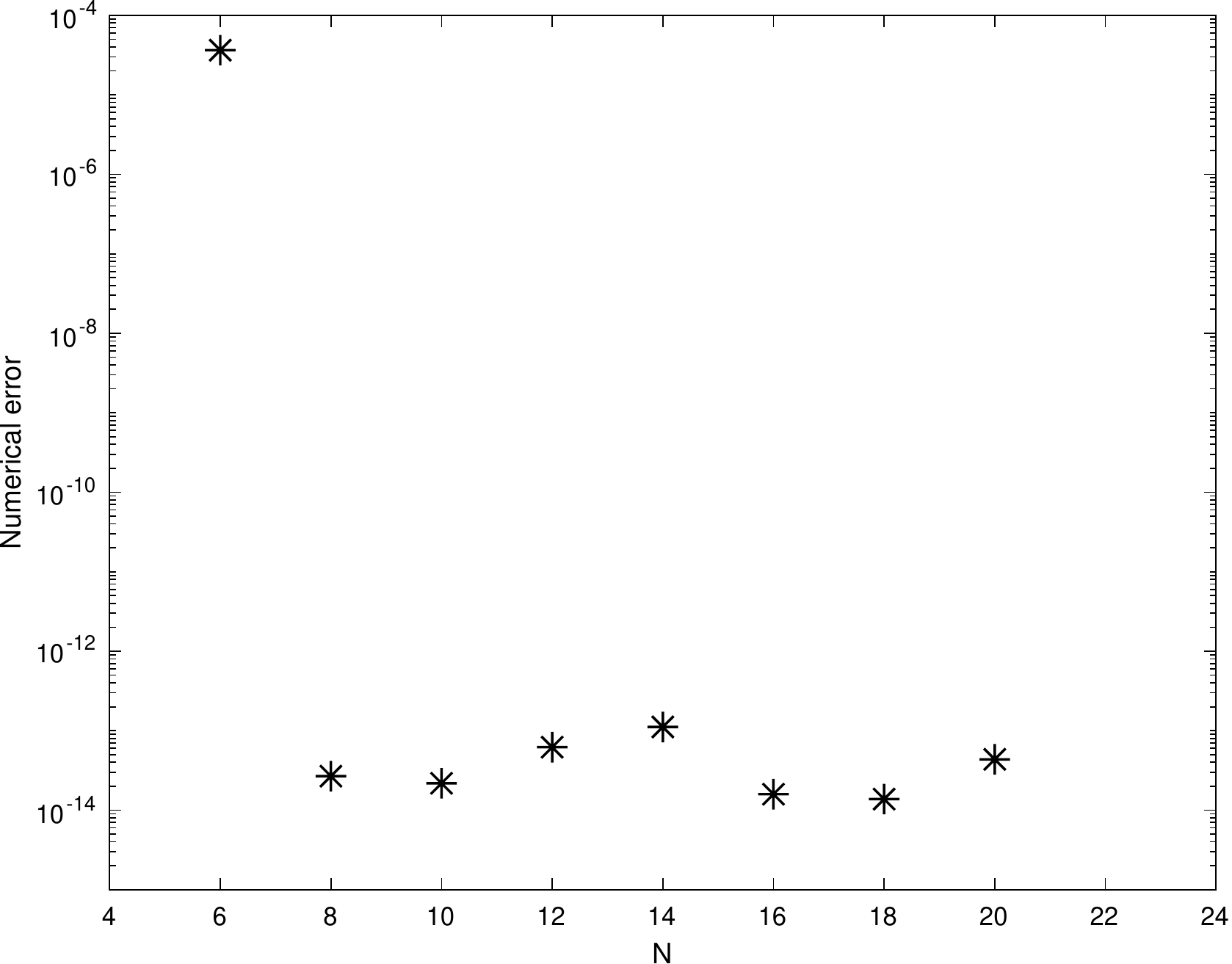}  \hskip 0.2cm
\includegraphics[height=2.8in,width=3in]{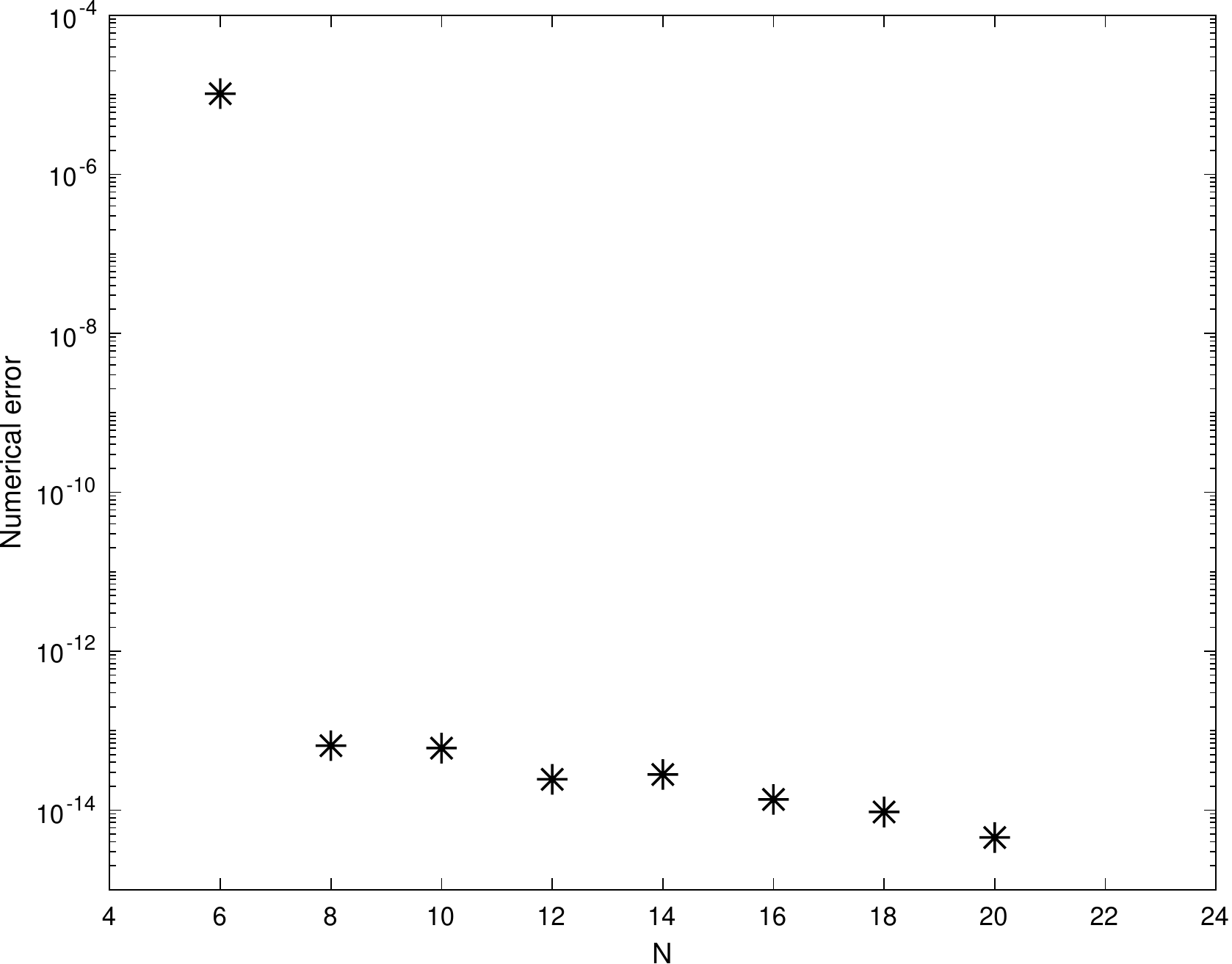} } 	
\caption{Left: Discrete $\ell^2$ numerical errors for the phase variable at $T=0.16$, plotted versus $N$, the number of spatial grid point, for the fully discrete pseudospectral scheme~\eqref{scheme-BDF-SPFC-1}. Right: The same numerical error plot for the other proposed scheme~\eqref{scheme-BDF-SPFC-1-alt}.  The time step size is fixed as $\dt = 10^{-4}$. An apparent spatial spectral accuracy is observed.}     
	\label{fig1}
	\end{figure}

To investigate the accuracy in space, we fix {\color{black}$\dt = 10^{-4}$, and absorb the temporal discretization errors into the external source term, so that the spatial approximation error dominates the overall numerical error.} We compute solutions with grid sizes {\color{black}$N = 6$ to $N=20$ in increments of 2}, and we solve up to time $T = 0.16$. The $\ell^2$ numerical errors, computed by the proposed numerical schemes~\eqref{scheme-BDF-SPFC-1} and \eqref{scheme-BDF-SPFC-1-alt}, are displayed in Fig.~\ref{fig1}. The spatial spectral accuracy is apparently observed for the phase variable. Due to the {\color{black}round-off errors}, a saturation of spectral accuracy appears with an increasing $N$, for both schemes.

	\begin{figure}
	\centering
        \hbox{ 
	\includegraphics[width=3in,height=2.8in]{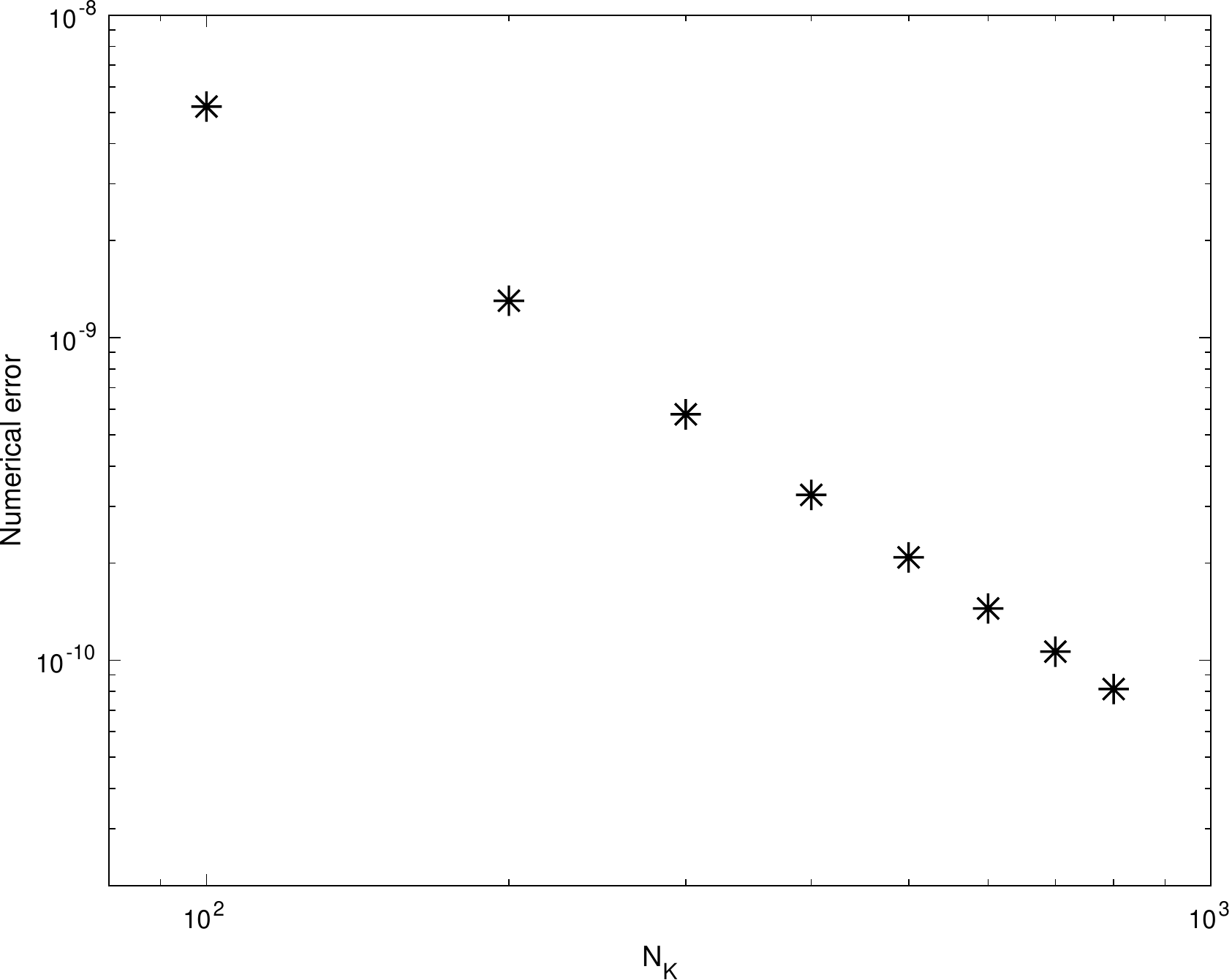}  \hskip 0.2cm
\includegraphics[height=2.8in,width=3in]{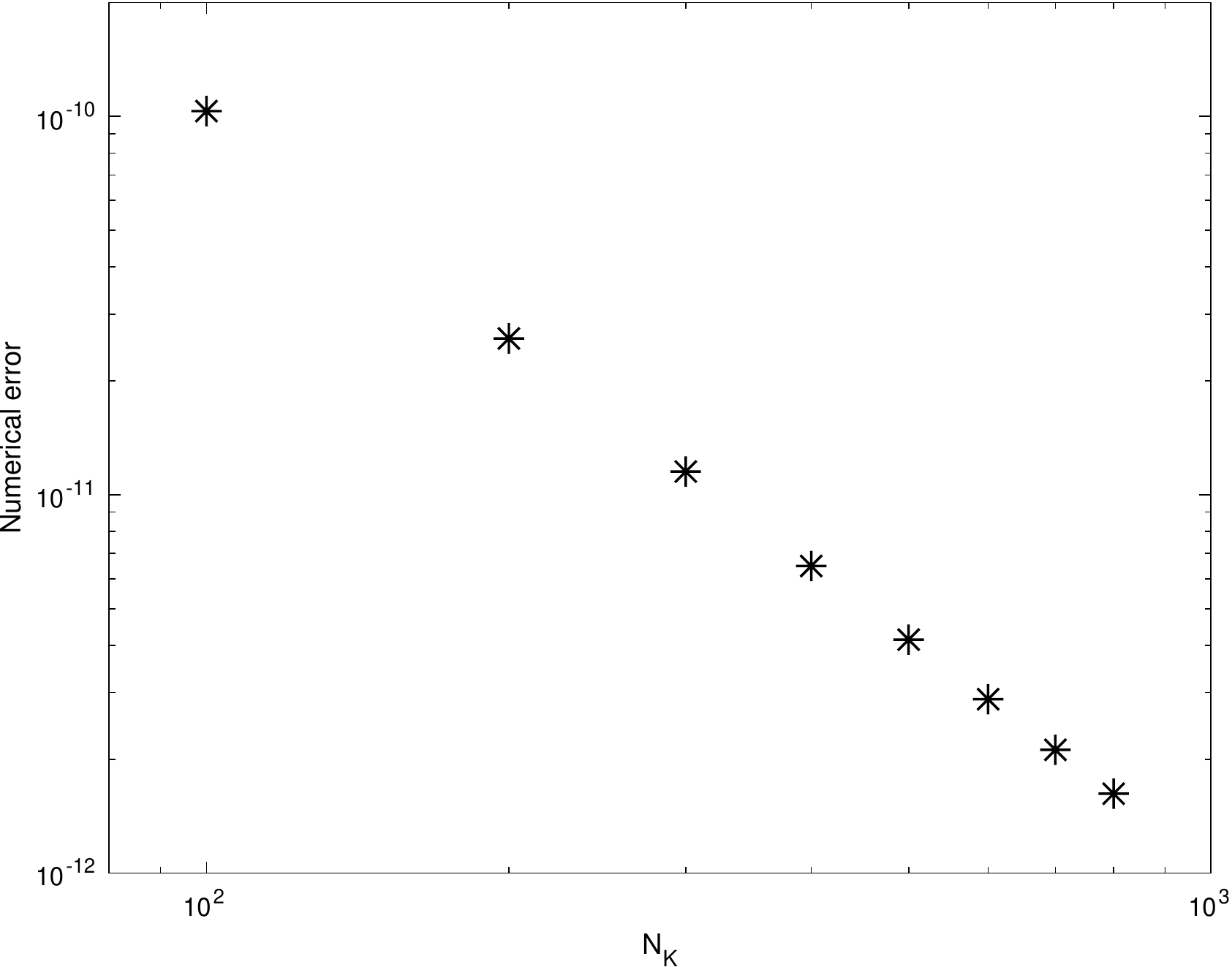} } 	
\caption{Left: Discrete $\ell^2$ numerical errors for the phase variable at $T=0.16$, plotted versus $N_K$, the number of time steps, for the fully discrete pseudo spectral scheme~\eqref{scheme-BDF-SPFC-1} (BDF2-ES-1). Right: The same numerical error plot for the second scheme~\eqref{scheme-BDF-SPFC-1-alt} (BDF2-ES-2).  Second-order temporal accuracy is observed. Note that the error for method 2, BDF2-ES-2, is significantly  smaller, as expected.}     
	\label{fig2}
	\end{figure}

To explore the temporal accuracy, we fix the spatial resolution as $N=128$ so that the numerical error is dominated by the temporal ones. We compute solutions with a sequence of time step sizes, $\dt = \frac{T}{N_k}$, with $N_k=100$ to $N_k=800$ in increments of 100, and the same final time $T=0.16$. Fig.~\ref{fig2} shows the discrete $\ell^2$ norms of the errors between the numerical and exact solutions, computed by the proposed numerical schemes~\eqref{scheme-BDF-SPFC-1} (BDF2-ES-1) and \eqref{scheme-BDF-SPFC-1-alt} (BDF2-ES-2), respectively. A clear second order accuracy has been demonstrated for both schemes. Moreover, while both schemes preserve second order temporal accuracy, it is interesting to observe that, the scale of the numerical error for the second scheme~\eqref{scheme-BDF-SPFC-1-alt} is lower than the first scheme~\eqref{scheme-BDF-SPFC-1}. Such a phenomenon comes from a subtle fact that, the second scheme~\eqref{scheme-BDF-SPFC-1-alt} takes a closer form than the standard BDF2 scheme. On the other hand, the standard BDF2 scheme gives smaller truncation errors than the ones with extrapolation and regularization terms, although a theoretical justification of its energy stability is not available.

\subsection{Numerical simulation of square symmetry patterns} 

The $4$-Laplacian term in \eqref{equation-SPFC} gives preference to rotationally invariant patterns with square symmetry. In this subsection, we perform two-dimensional numerical simulations showing the emergence of these patterns. The rest of the parameters are given by $a = 0.5$ and $\Omega =(0,L)^2$, with $L=100$. The initial data for the simulations are given by 
\begin{equation}
	\label{eqn:init2}
 \phi^0_{i,j}=0.05\cdot(2r_{i,j}-1),
	\end{equation}
where the $r_{i,j}$ are uniformly distributed random numbers in $[0, 1]$. In addition, 
we add nucleation sites at specific locations in the domain, with magnitude 10, at (50,50) as an example of one nucleation site, at (25,25), (25,75), (75,25), and (75,75) as another example of four nucleation sites. For the temporal step size $\dt$, we use increasing values of $\dt$ in the time evolution: $\dt = 0.05$ on the time interval $[0,1000]$ and $\dt = 0.1$ on the time interval $[1000, 9000]$. Whenever a new time step size is applied, we initiate the two-step numerical scheme by  taking $\phi^{-1} = \phi^0$, with the initial data $\phi^0$ given by the final time output of the last time period. The time snapshots of the evolution by using the given parameters are presented in Figures~\ref{fig:long-time-spfc-one} (one nucleation site) and \ref{fig:long-time-spfc-four} (four nucleation sites). These tests confirm the emergence of the rotationally invariant square-symmetry patterns in the density field.

	\begin{figure}[h]
	\begin{center}
	\begin{subfigure}{0.48\textwidth}
\includegraphics[height=0.48\textwidth,width=0.48\textwidth]{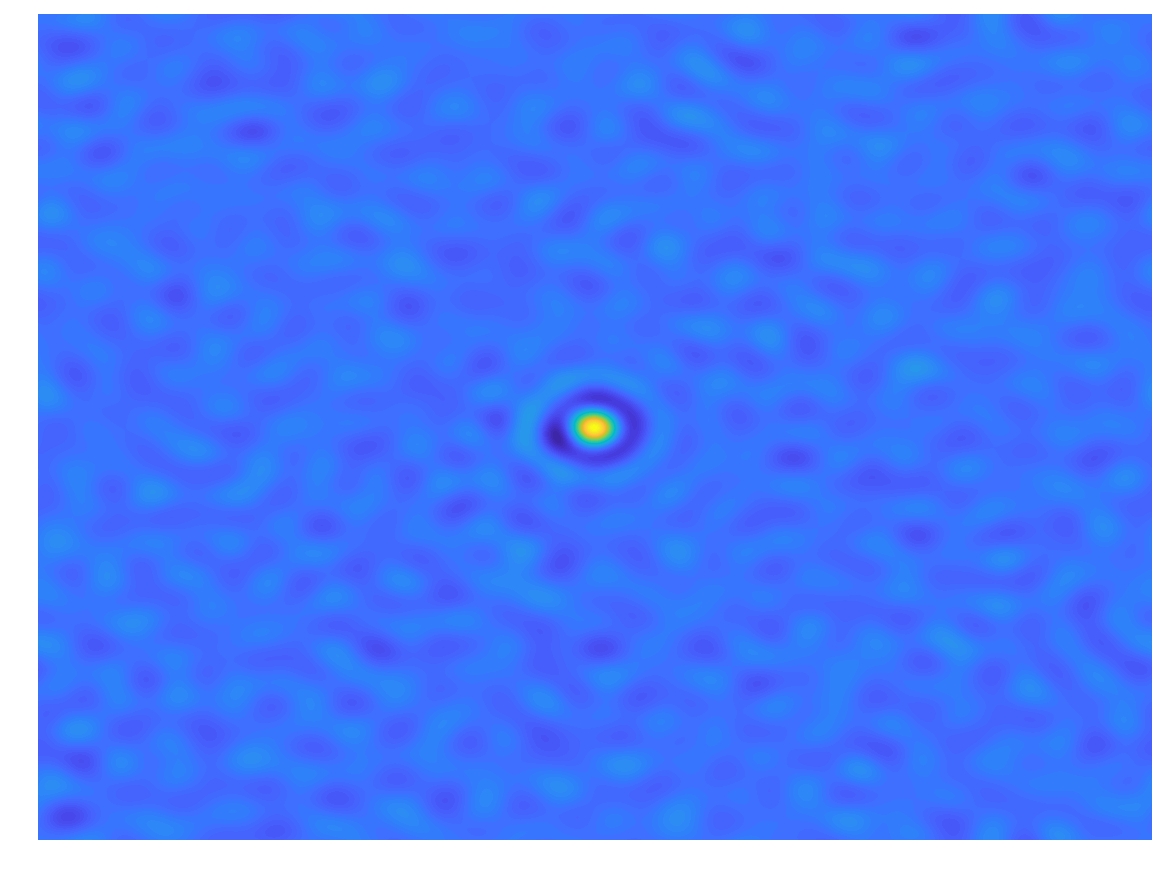} 
\includegraphics[height=0.48\textwidth,width=0.48\textwidth]{phi_T2_t1.jpg} 		
\caption*{$t=1, 10$}
	\end{subfigure}

	\begin{subfigure}{0.48\textwidth}
\includegraphics[height=0.48\textwidth,width=0.48\textwidth]{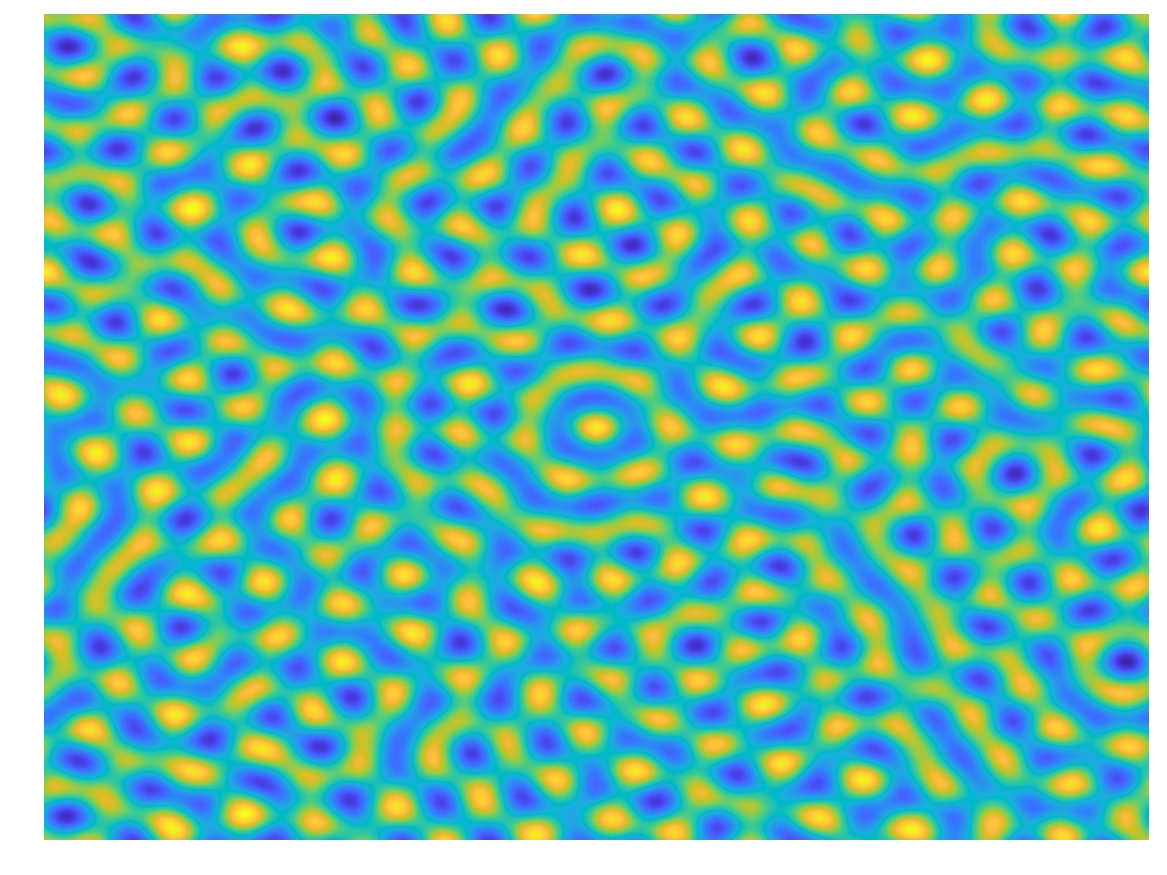} 
\includegraphics[height=0.48\textwidth,width=0.48\textwidth]{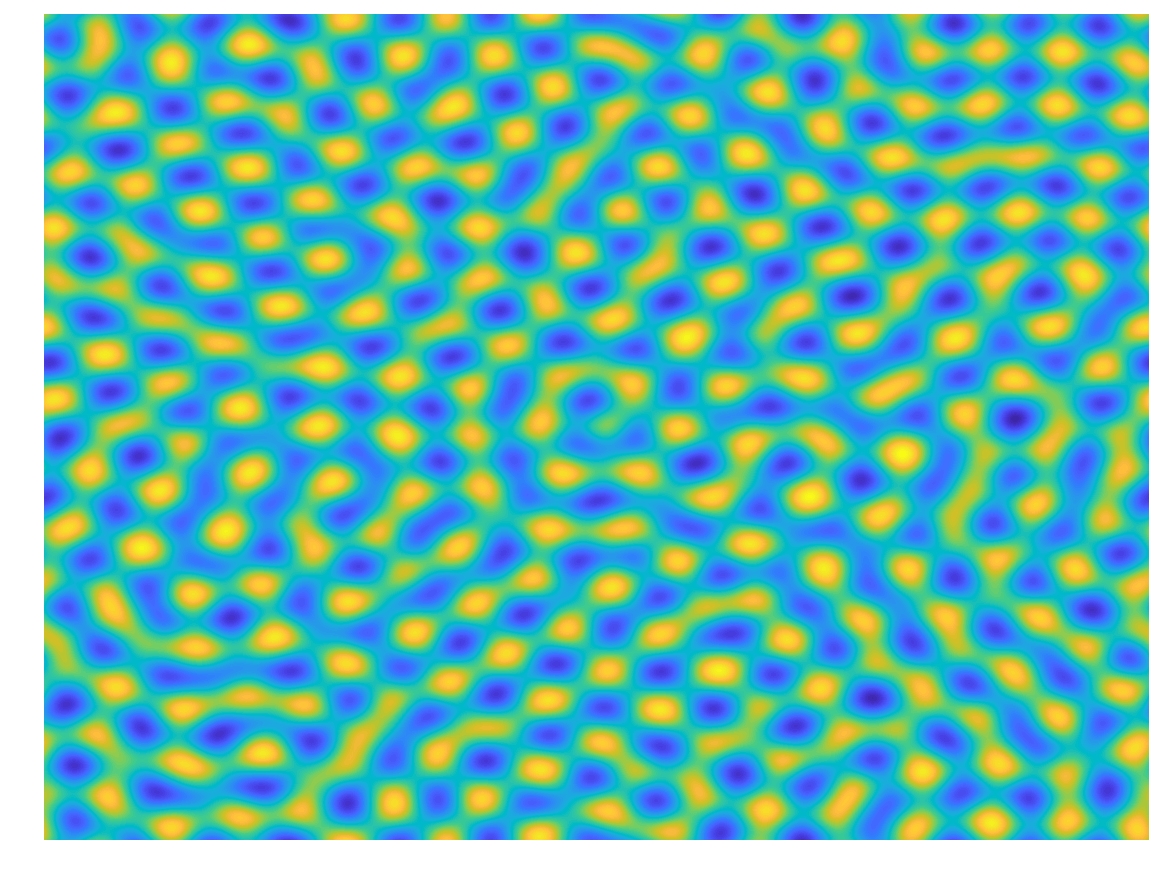}
\caption*{$t=20, 40$}
	\end{subfigure}
	
	\begin{subfigure}{0.48\textwidth}
\includegraphics[height=0.48\textwidth,width=0.48\textwidth]{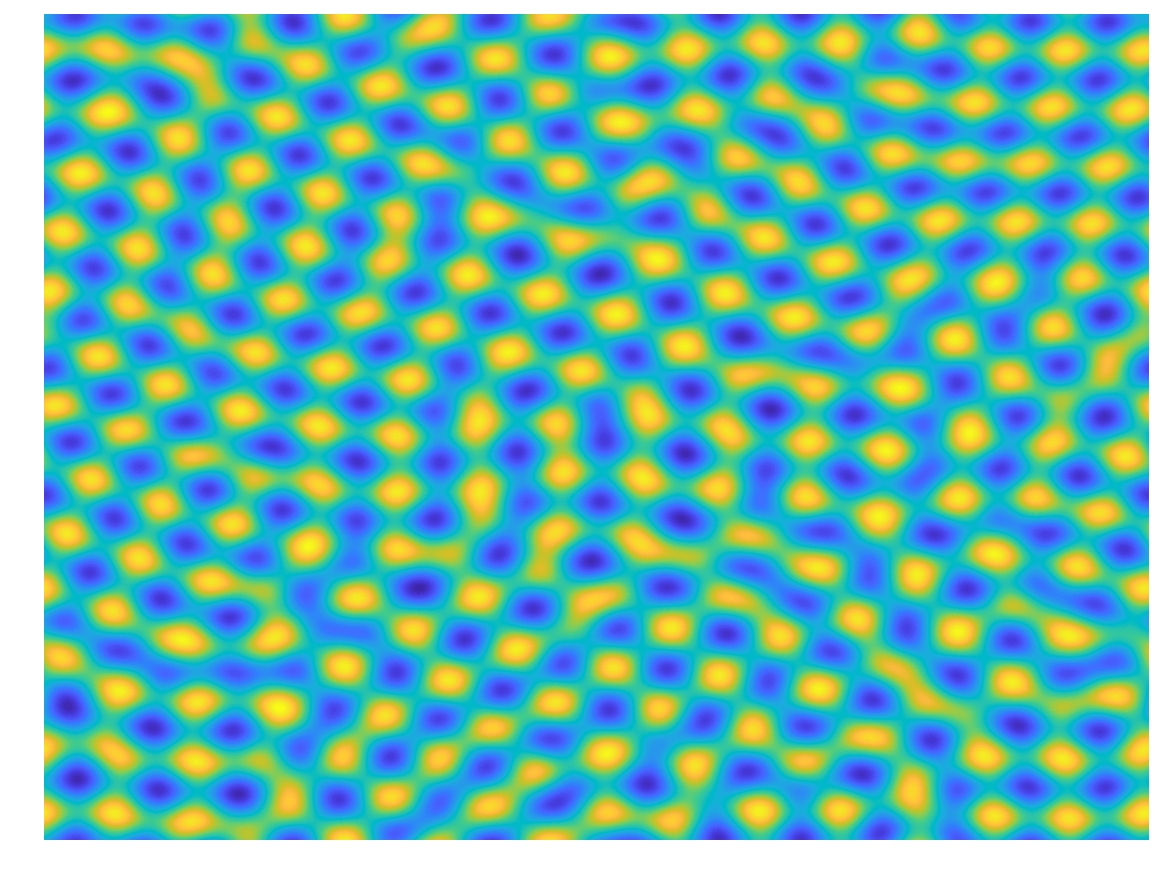}
\includegraphics[height=0.48\textwidth,width=0.48\textwidth]{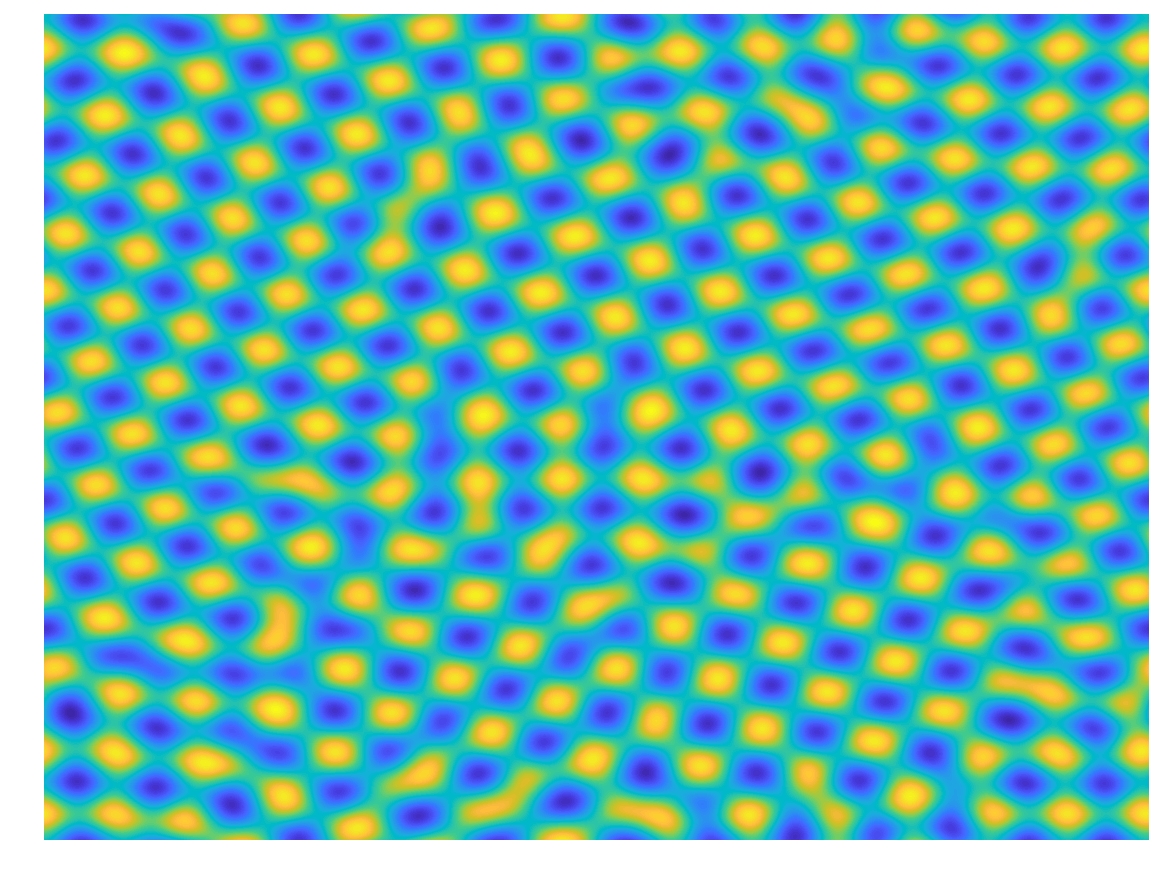}
\caption*{$t=100,200$}
	\end{subfigure}
	
	\begin{subfigure}{0.48\textwidth}
\includegraphics[height=0.48\textwidth,width=0.48\textwidth]{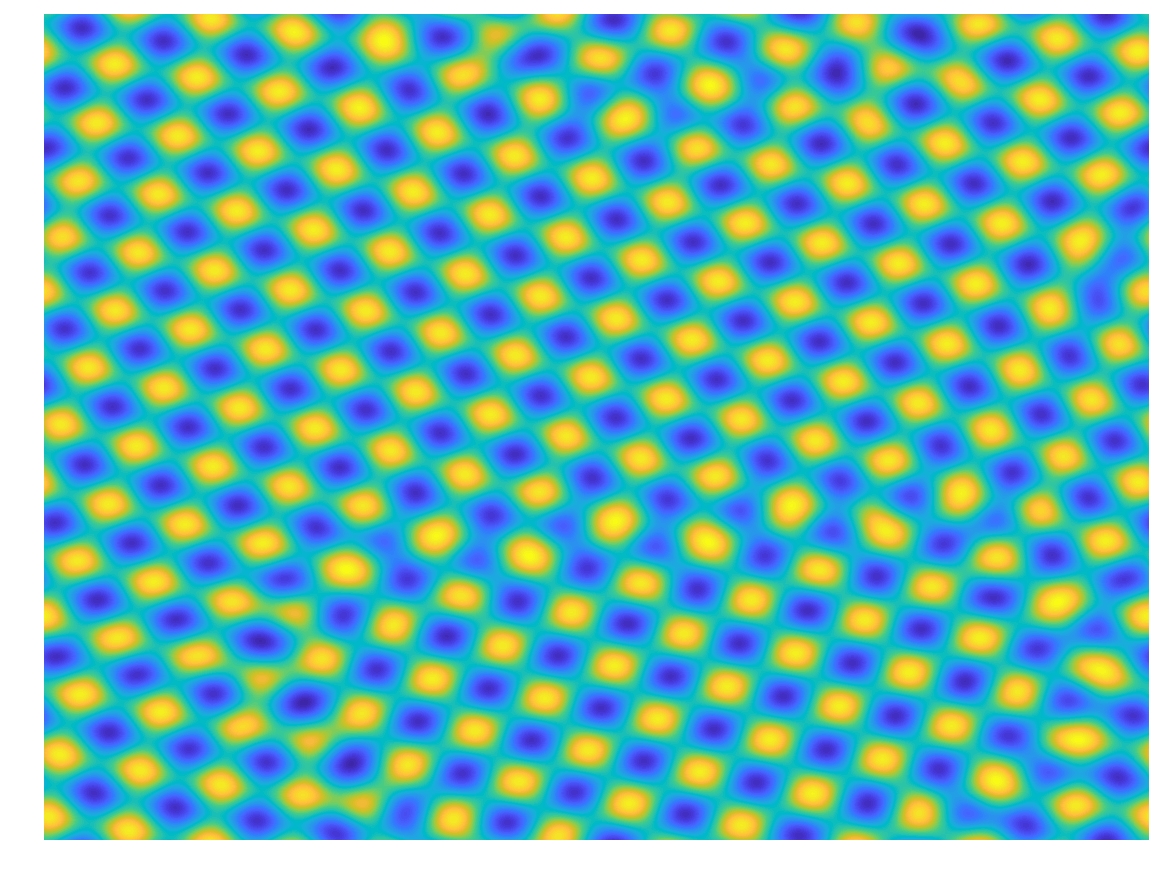} 
\includegraphics[height=0.48\textwidth,width=0.48\textwidth]{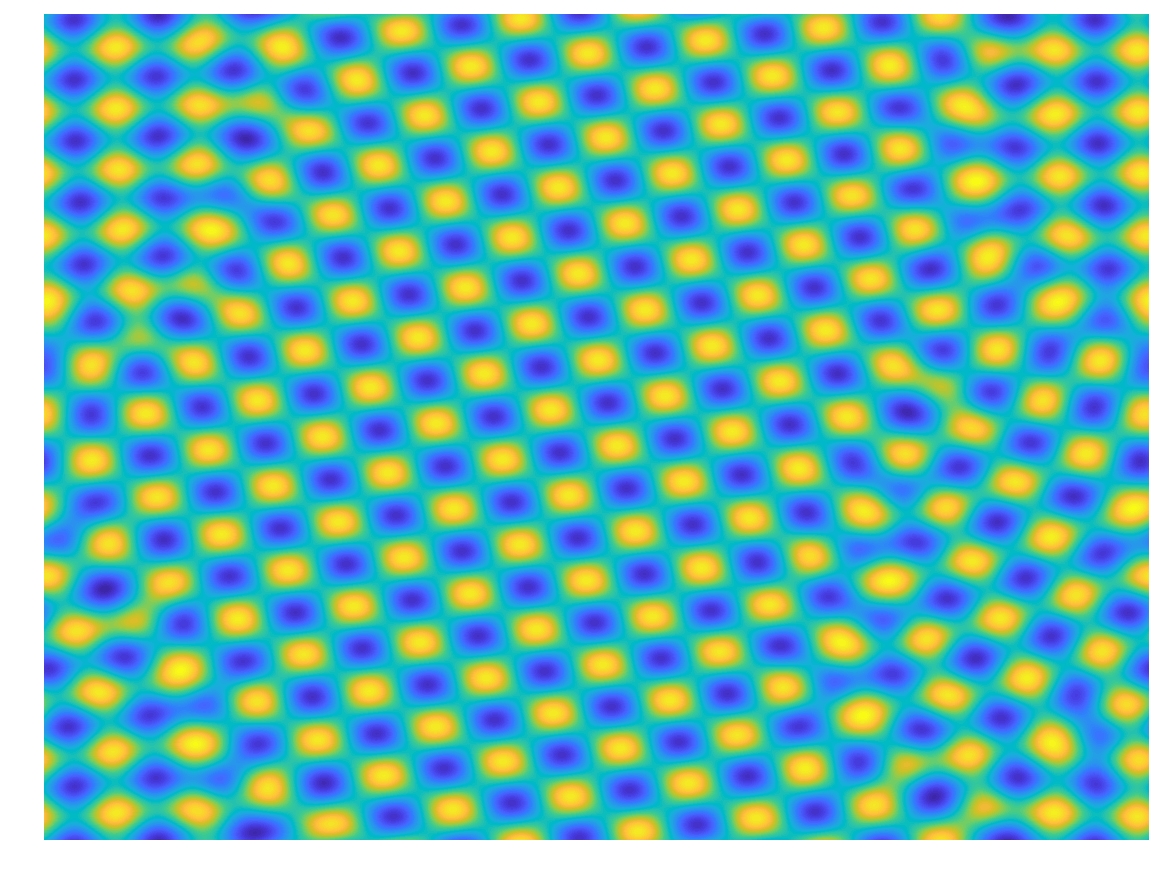}		
\caption*{$t=3000, 9000$}
	\end{subfigure}
	
\caption{Time snapshots of the evolution for square phase field crystal model, with one nucleation site at $(50, 50)$. The time sequence for the snapshots is $t=1$, 10, 20, 40,  100, 200, 3000, and 9000. The parameters are $a =0.5, \Omega=[0, 100]^2$.}
	\label{fig:long-time-spfc-one}
	\end{center}
	\end{figure}

	\begin{figure}[h]
	\begin{center}
	\begin{subfigure}{0.48\textwidth}
\includegraphics[height=0.48\textwidth,width=0.48\textwidth]{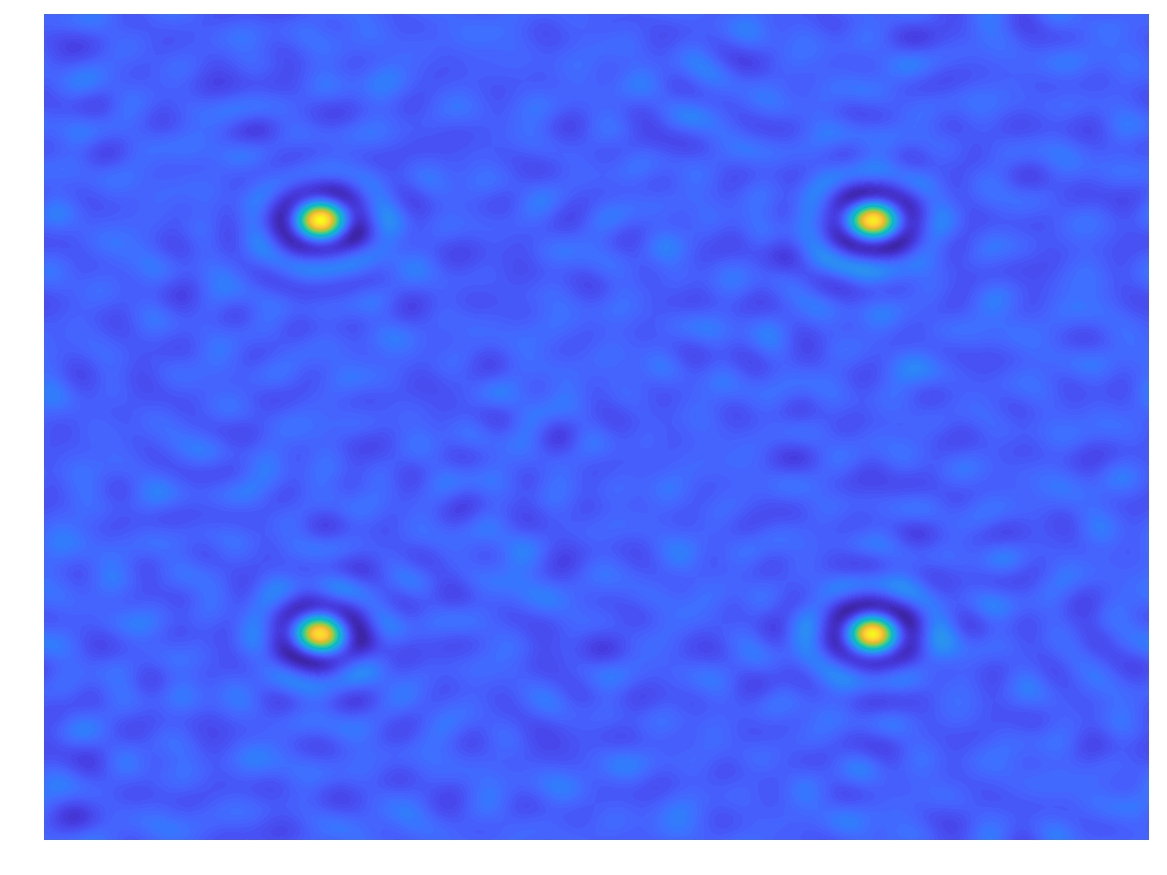} 
\includegraphics[height=0.48\textwidth,width=0.48\textwidth]{phi_T4_t1.jpg} 
\caption*{$t=1, 10$}
	\end{subfigure}
	
	\begin{subfigure}{0.48\textwidth}
\includegraphics[height=0.48\textwidth,width=0.48\textwidth]{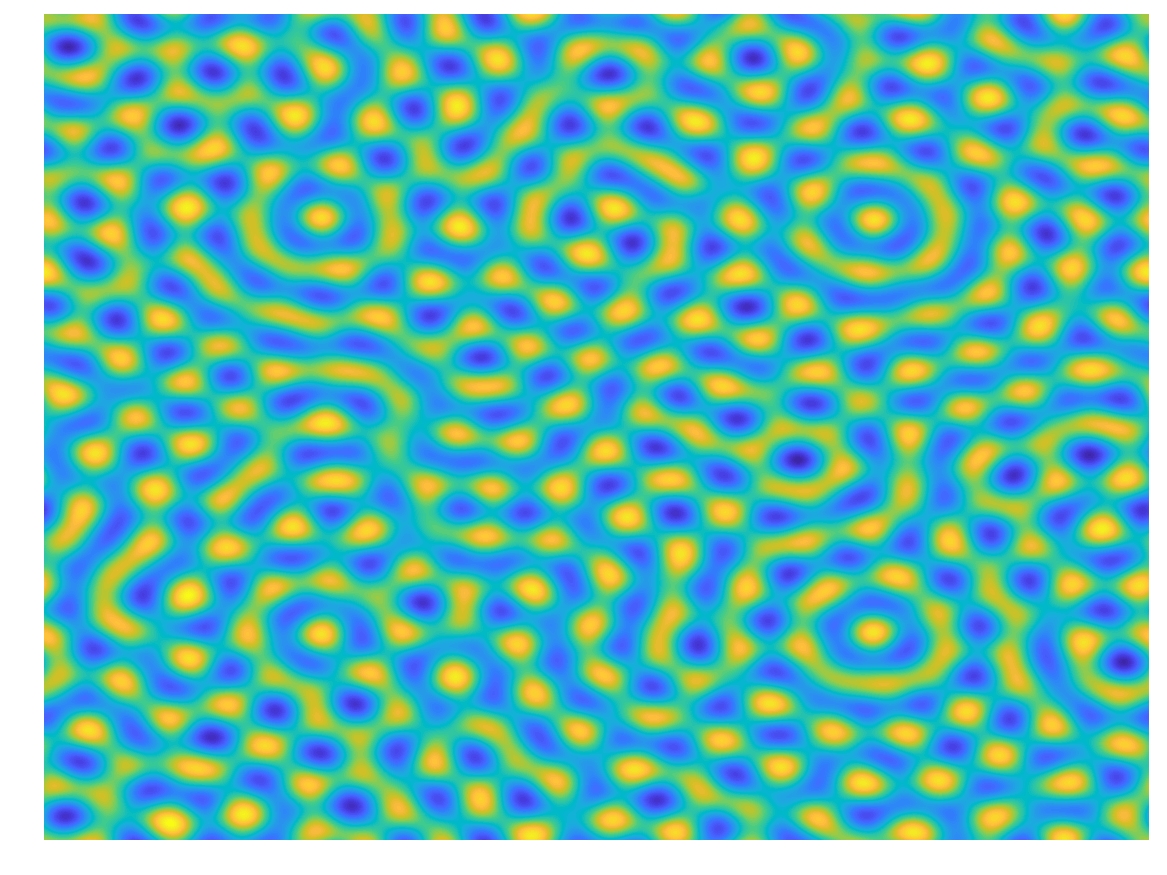} 		\includegraphics[height=0.48\textwidth,width=0.48\textwidth]{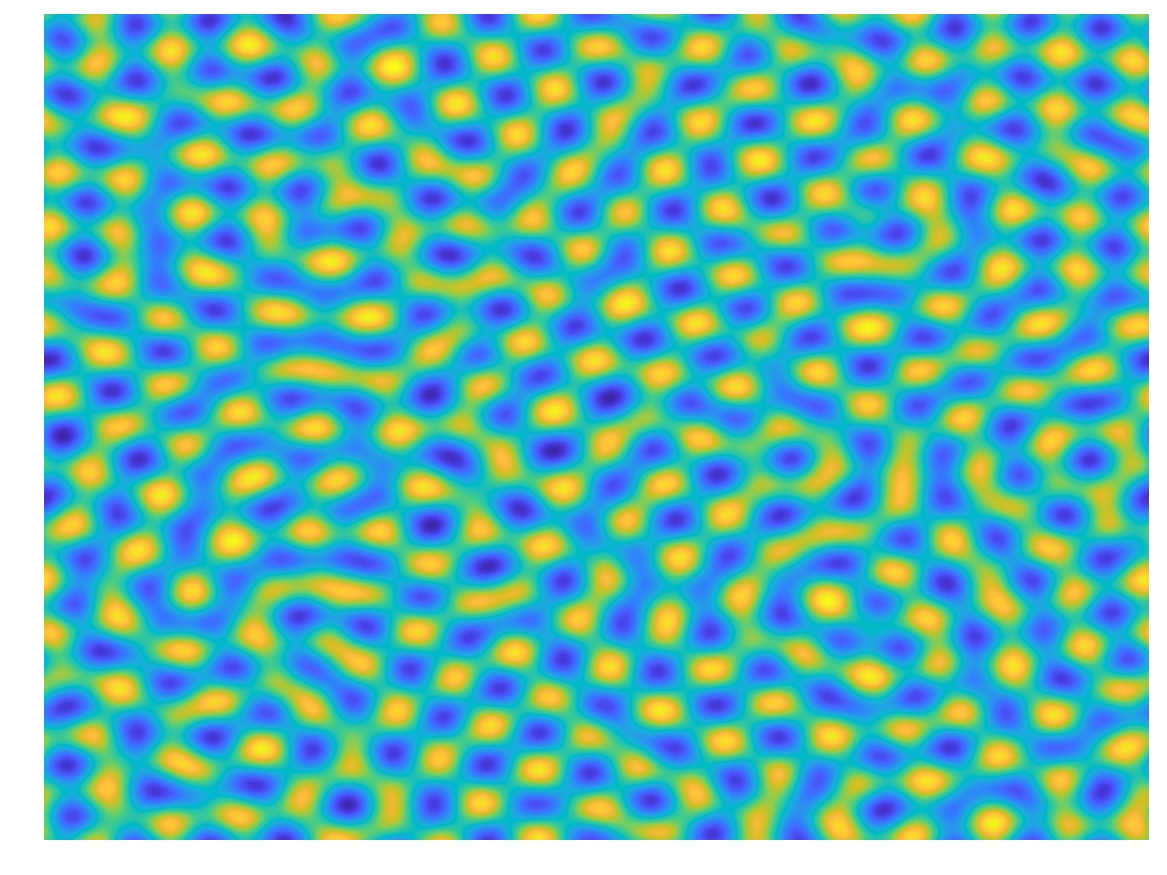}
\caption*{$t=20, 40$}
	\end{subfigure}
	
	\begin{subfigure}{0.48\textwidth}
	\includegraphics[height=0.48\textwidth,width=0.48\textwidth]{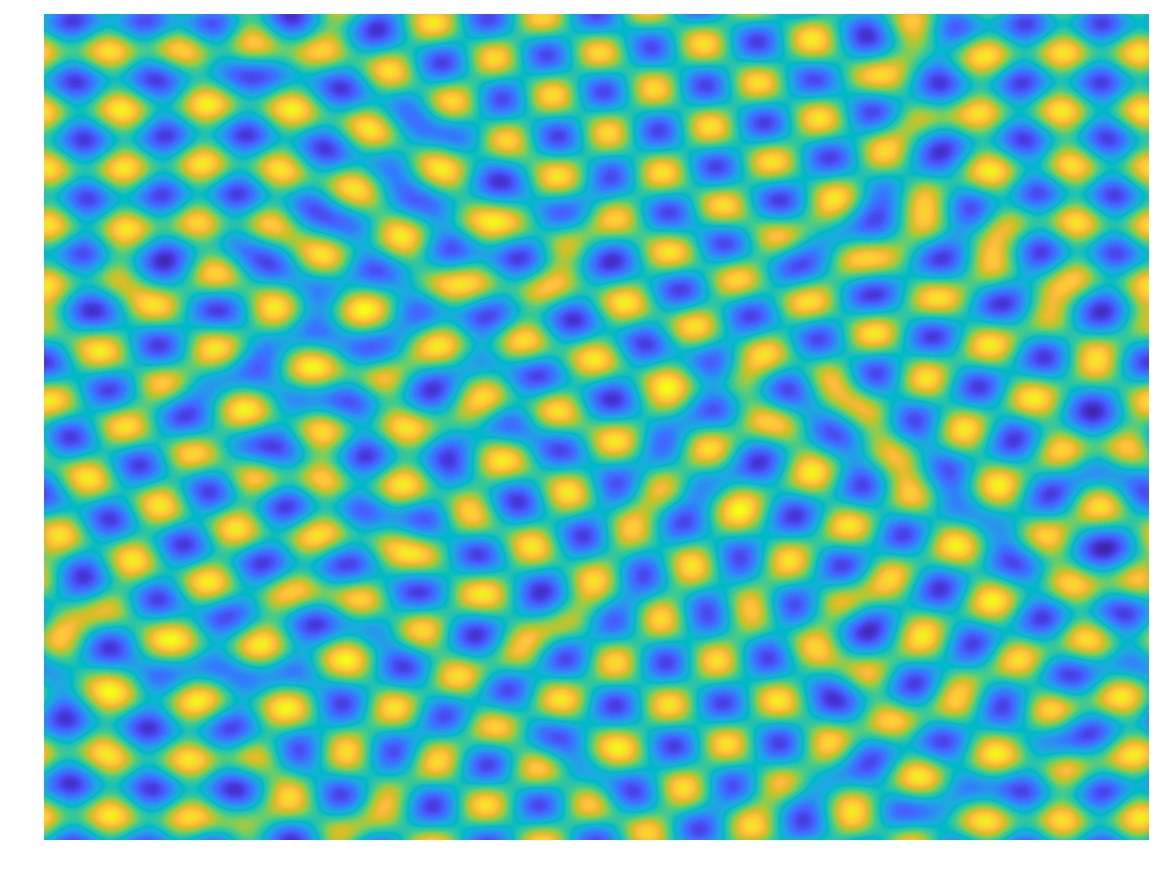}			\includegraphics[height=0.48\textwidth,width=0.48\textwidth]{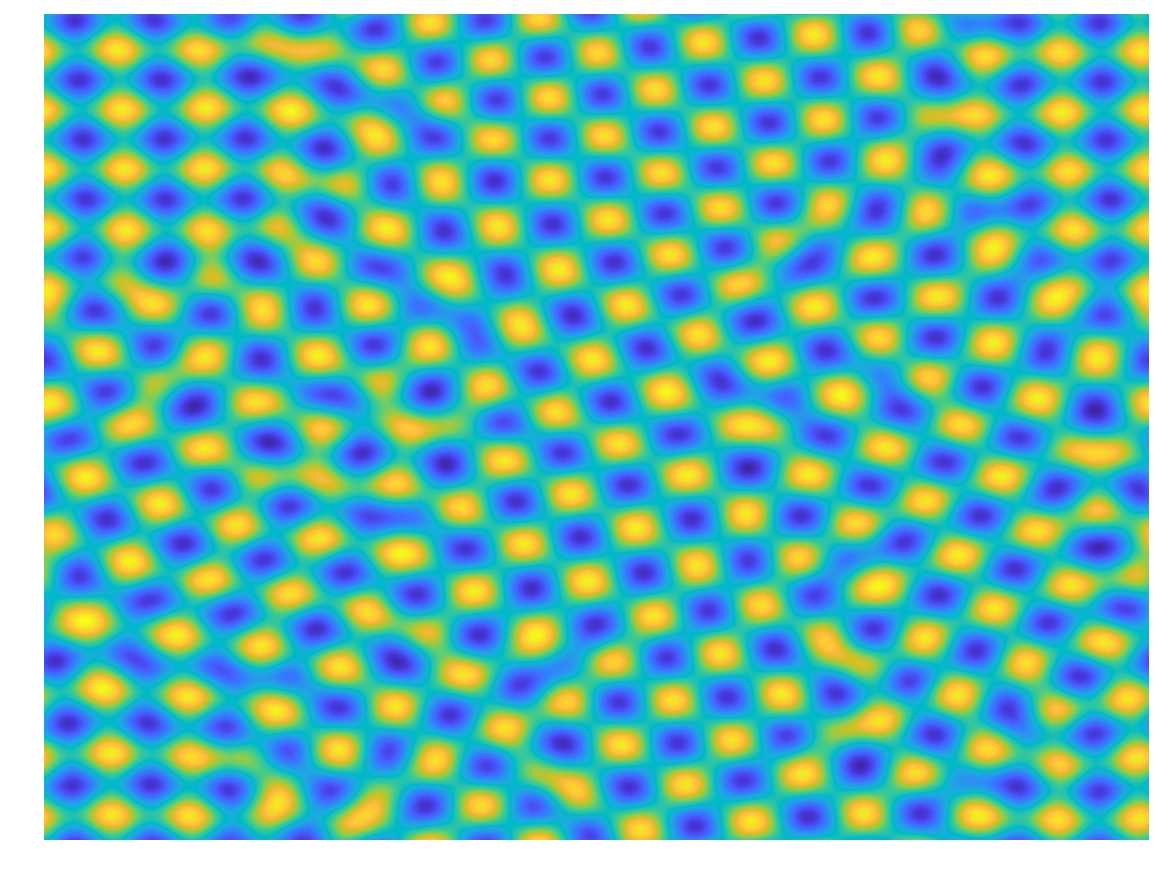}
\caption*{$t=100,200$}
	\end{subfigure}
	
	\begin{subfigure}{0.48\textwidth}
\includegraphics[height=0.48\textwidth,width=0.48\textwidth]{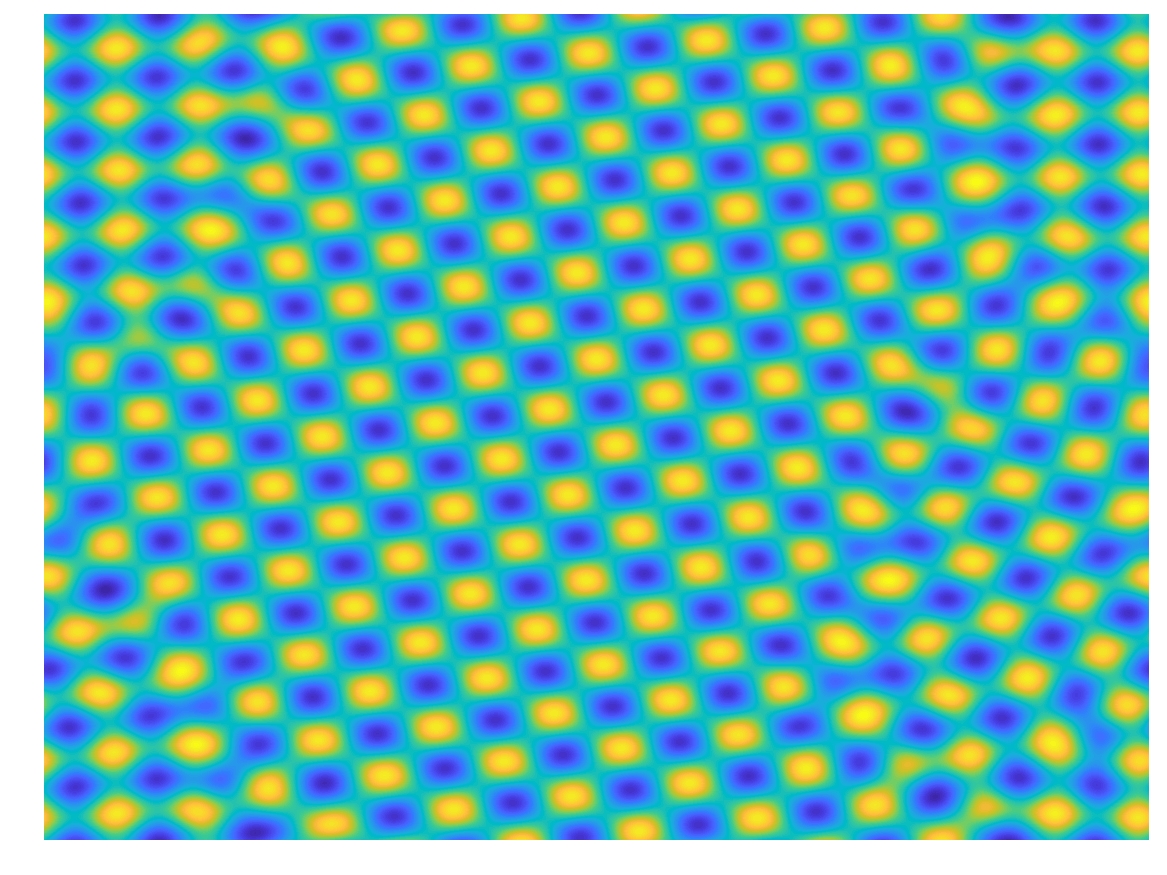} 
\includegraphics[height=0.48\textwidth,width=0.48\textwidth]{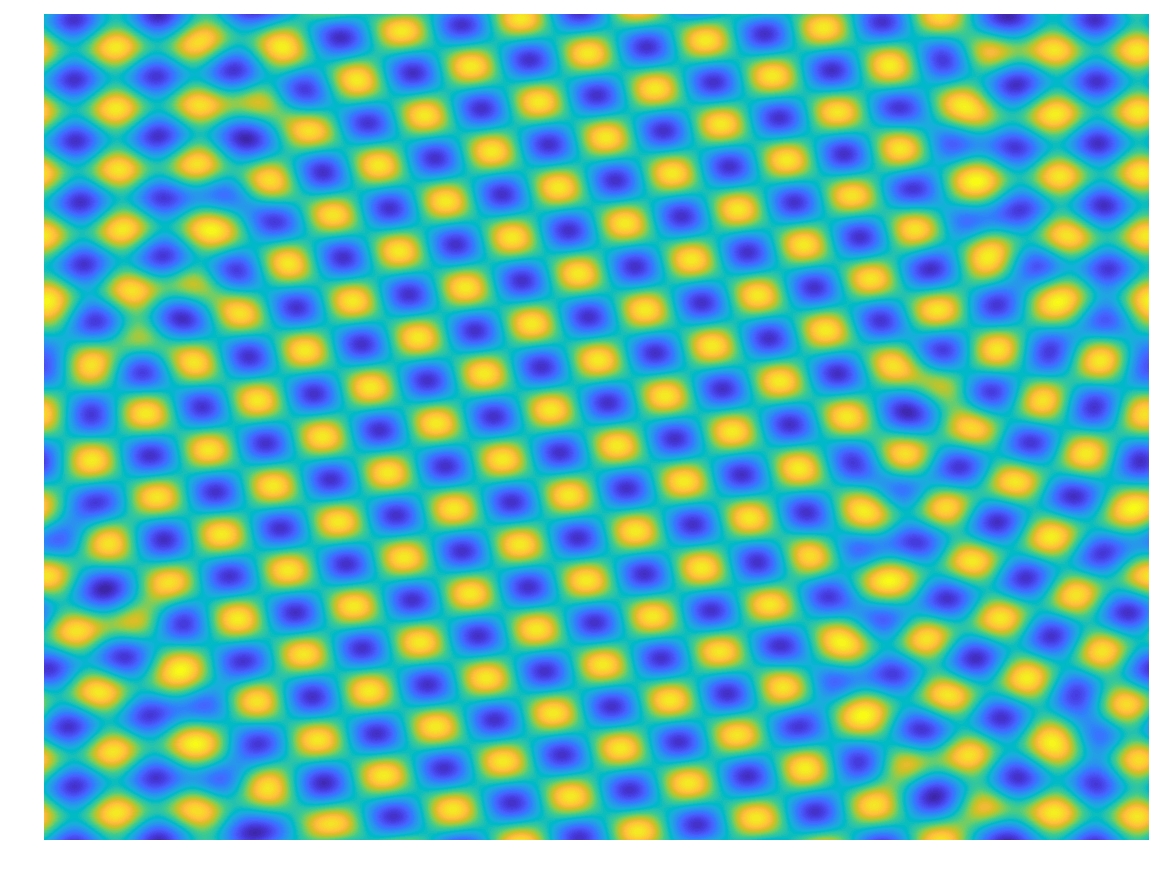}
\caption*{$t=3000, 9000$}
	\end{subfigure}
	
\caption{Time snapshots of the evolution for squared phase field crystal model, with four nucleation sites at $(25, 25), (25,75), (75,25), (75,75)$, respectively. The time sequence for the snapshots is  $t=1$, 10, 20, 40, 100, 200, 500 3000, and 9000. The parameters are $a =0.5, \Omega=[0, 100]^2$.}
	\label{fig:long-time-spfc-four}
	\end{center}
	\end{figure}
	
{\color{black}To illustrate the energy stability property of the proposed numerical scheme, we display the energy evolution of the one nucleation site example, up to $t=1000$, in Figure~\ref{fig:energy evolution}. Three different time step sizes, $\dt_1 = 0.1$, $\dt_2 =0.05$ and $\dt_3 = 0.025$, have been taken in the comparison. The energy dissipation property is satisfied in the numerical simulation. Moreover, the numerical results have a perfect agreement between different time step sizes; this is an amazing fact, due to the large time scale of the numerical simulations.}  

	\begin{figure}
	\centering 
	\includegraphics[width=4.0in]{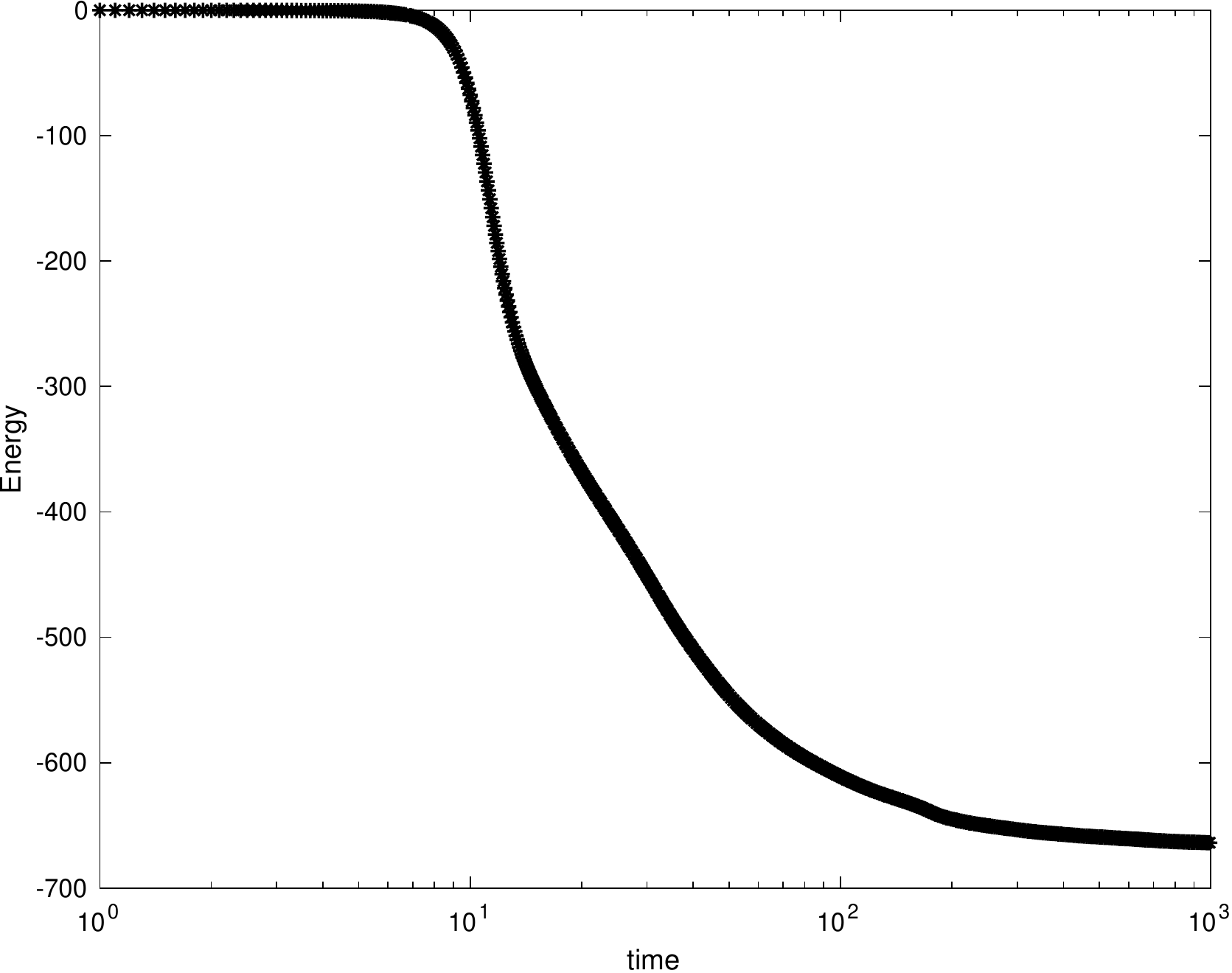}  
\caption{Semi-log plot of the temporal evolution the energy $E_N$ up to $t=1000$, with three different time step sizes: $\dt_1 = 0.1$, $\dt_2 =0.05$ and $\dt_3 = 0.025$. The dotted, solid and star lines represent the plots for $\dt_3 = 0.025$, $\dt_2 = 0.05$ and $\dt_1 = 0.1$, respectively. The plots overlap so that differences are indistinguishable.}     
	\label{fig:energy evolution}
       \end{figure}

\section{Concluding remarks} \label{sec:conclusion}

In this article, we proposed and analyzed two energy stable Fourier pseudo-spectral schemes for the square phase field crystal (SPFC) equation, a gradient flow to model the crystal growth. The schemes exhibit  second order temporal accuracy and spectral accuracy in space. In the energy functional expansion and the corresponding PDE formulations, a composition of the 4-Laplacian and the regular Laplacian operators makes the physical system very challenging, at both the theoretical and numerical levels. To overcome this well-known difficulty, we come up with a modified BDF scheme, with the second order BDF stencil applied in the time direction, combined with an appropriate extrapolation for the concave diffusion term, to ensure the unique solvability and energy stability. In particular, a second order artificial Douglas-Dupont regularization term is added to ensure the energy stability, and a careful treatment leads to the regularization diffusion at a lower order than the surface diffusion term. Such a subtle analysis avoids a higher order artificial diffusion term, therefore a reduced numerical dissipation is expected for the numerical effect. At the theoretical side, the unique solvability, energy stability could be proved with the help of the summation-by-parts formulas in the Fourier pseudo-spectral space. In addition, the energy stability yields a uniform in time $H_N^2$ bound for the numerical solution, and an aliasing error control technique enables us to derive a discrete Sobolev embedding from $H^2$ into $W^{1,6}$. As a result of such a nonlinear estimate, we are able to derive an optimal rate error analysis in the $\ell^\infty (0,T; \ell^2) \cap \ell^2 (0,T; H_N^3)$ norm. In the numerical implementation, the preconditioned steepest descent (PSD) iteration is needed to deal with the composition of the highly nonlinear 4-Laplacian term and the standard Laplacian term, and a geometric convergence could be proved for this iteration. It is the first such result for a 4-Laplacian solver in an $H^{-1}$ gradient flow. A few numerical experiments are presented to demonstrate the efficiency and accuracy of the proposed scheme, including the numerical accuracy test and numerical simulations of square symmetry patterns, with one nucleation site and four nucleation sites, respectively.

	\section*{Acknowledgements} 
This work is supported in part by the Longshan Talent Project of SWUST 18LZX529  (K. Cheng), the grants NSF DMS-1418689 (C.~Wang), NSF DMS-1418692 and NSF DMS-1719854 (S.~Wise). 

	\appendix

\section{Proof of Proposition~\ref{prop:embedding}}
	\label{proof:Prop 2.2}
	
For simplicity of presentation, in the analysis of $\| \nabla_N f \|_6$, we are focused on the estimate of $\| {\cal D}_x f \|_6$. Due to the periodic boundary condition for $f$ and its cell-centered representation, it has a corresponding discrete Fourier transformation, as the form given by~\eqref{spectral-coll-1}: 
	\begin{eqnarray}
f_{i,j,k} = \sum_{\ell,m,n=-K}^{K} \hat{f}_{\ell,m,n}^N \exp \left( 2 \pi {\rm i} ( \ell x_i + m y_j + n z_k ) \right) .
   \label{def:Fourier-1}
	\end{eqnarray}
Then we make its extension to a continuous function:
	\begin{equation}
	\label{def:extension-1}
f_N (x,y,z) = \sum^{K}_{\ell,m,n=-K} \hat{f}^N_{\ell,m,n} \exp \left( 2 \pi {\rm i} ( \ell x + m y + n z ) \right)  .
	\end{equation}
	
The following result is excerpted as Lemma A.2 in~\cite{fengW17c}; similar analyses have also been reported in recent works \cite{cheng16a, feng2017preconditioned}, \emph{et cetera}.  

\begin{lem} \label{lem:A.2} \cite{fengW17c} 
  For $g \in \mathcal{G}_N$,  we have
\begin{eqnarray}
  \| g \|_p \le \sqrt{\frac{p}{2}} \| g_N \|_{L^p} ,  \quad \mbox{with $p = 4, 6$} .  \label{lemma A.2-0}
\end{eqnarray}
\end{lem}

Then we proceed into the proof of Proposition~\ref{prop:embedding}. 

\begin{proof} 
We denote a discrete grid function, $g := {\cal D}_x f$, at a point-wise level. Since $f$ corresponds to $f_N \in {\cal B}^K$ (the space of trigonometric polynomials of degree at most $K$), we see that $g_N = \partial_x f_N$. As a result, an application of the preliminary estimate~\eqref{lemma A.2-0} reveals that 
\begin{eqnarray} 
  \|  {\cal D}_x f \|_6 = \| g \|_6 \le \sqrt{3} \| g_N \|_{L^6} 
  = \sqrt{3} \| \partial_x f_N \|_{L^6}  \le C \| \Delta f_N \|_{L^2} 
  = C \| \Delta_N f \|_2 ,  \label{prop embedding-1}
\end{eqnarray}  
in which the fourth step is based on the Sobolev embedding in the continuous space: $\| \partial_x f_N \|_{L^6}  \le C \| \Delta f_N \|_{L^2}$, and the last step comes from the fact that $f$ uniquely corresponds to $f_N$. Similar estimates in the other directions could be derived at the same manner: 
\begin{eqnarray} 
  \|  {\cal D}_y f \|_6 \le C \| \Delta_N f \|_2 ,  \quad 
  \|  {\cal D}_z f \|_6 \le C \| \Delta_N f \|_2  .  \label{prop embedding-2}
\end{eqnarray}  
This completes the proof of Proposition~\ref{prop:embedding}. 
\end{proof}

	\bibliographystyle{plain}
	\bibliography{spfc.bib}

\end{document}